\documentclass{amsart}[11pt]
\usepackage{amssymb, amscd, stmaryrd}
\usepackage{enumerate}
\usepackage[dvips,dvipdf]{graphicx}
\usepackage[usenames,dvipsnames]{color}

\usepackage{color}
\usepackage{color, colortbl}
\usepackage{multirow}
\usepackage[table]{xcolor}
\usepackage{comment}

\newcommand{\pa}{\partial}

\newcommand{\maT}{\mathcal T}

\newcommand{\maE}{\mathcal E}

\newcommand{\maL}{\mathcal L}

\newcommand{\maO}{\mathcal O}

\newcommand{\be}{\begin{eqnarray}}
\newcommand{\ee}{\end{eqnarray}}
\newcommand{\ben}{\begin{eqnarray*}}
\newcommand{\een}{\end{eqnarray*}}

\newcommand{\maV}{\mathcal V}

\newcommand{\bkappa}{\boldsymbol {\kappa}}

\newcommand{\bsigma}{\boldsymbol {\sigma}}

\newcommand{\bmu}{\boldsymbol {\mu}}
\newcommand{\bbeta}{\boldsymbol {\eta}}

\newcommand{\bone}{\textit{\textbf 1}}
\newcommand{\bzero}{\textit{\textbf 0}}
\newcommand{\ba}{\textit{\textbf a}}
\newcommand{\bb}{\textit{\textbf b}}

\newtheorem{theorem}{Theorem}[section]

\newtheorem{corollary}[theorem]{Corollary}

\newtheorem{lemma}[theorem]{Lemma}

\theoremstyle{definition}
\newtheorem{definition}[theorem]{Definition}
\theoremstyle{definition}
\newtheorem{algorithm}[theorem]{Algorithm}
\theoremstyle{remark}
\newtheorem{remark}[theorem]{Remark}

\author[H. Li]{Hengguang Li} \address{Hengguang Li, Department of Mathematics, Wayne State University, Detroit, MI 48202, USA}
\email{hli@math.wayne.edu}

\thanks{H. Li was partially supported by the NSF Grant  DMS-1418853 and by the Wayne State University Grants Plus Program.}

\begin{document}
\title[new anisotropic finite element methods on polyhedral domains]{A new anisotropic finite element method on polyhedral domains: interpolation error analysis}


\begin{abstract} On a  polyhedral domain $\Omega\subset\mathbb R^3$, consider the Poisson equation with the Dirichlet boundary condition. For  singular solutions from the non-smoothness of the domain boundary, we propose new anisotropic mesh refinement algorithms to improve the convergence of finite element approximation. The proposed algorithm is simple, explicit, and requires less geometric conditions on the mesh and on the domain. Then,  we develop interpolation error estimates in suitable weighted spaces for the anisotropic mesh, especially for the tetrahedra violating the maximum angle condition. These estimates can be used to design optimal finite element methods approximating singular solutions. We report numerical test results to validate the method.
\end{abstract}

\maketitle

\section{introduction}

Let $\Omega\subset \mathbb R^3$ be  a bounded polyhedral domain. Consider the Poisson equation with the  Dirichlet boundary condition, 
\begin{eqnarray}\label{eqn.n1}
-\Delta u=f\quad\text{in}\quad \Omega,  \qquad u=0\quad\text{on}\quad\partial \Omega.
\end{eqnarray}
The regularity of the solution  is determined by the smoothness of the boundary $\pa\Omega$ and the smoothness of the given data $f$. For example, when the domain has a smooth boundary, the solution continuously depends on the given data $f$ in Sobolev spaces with the full regularity estimate \cite{Evans98, LM172}
\begin{eqnarray}\label{eqn.fulll}
\|u\|_{H^{m+1}(\Omega)}\leq C\|f\|_{H^{m-1}(\Omega)},\qquad m\geq 0.
\ee
On domains with a non-smooth boundary, equation (\ref{eqn.n1})  usually possesses solutions with singularities near the non-smooth points, and therefore the estimate in (\ref{eqn.fulll}) no longer holds, even when $f$ is smooth. The lack of regularity  in the solution can cause severe  convergence issues in the numerical approximations \cite{CD02, CDS05,Wahlbin84}. 

Addressing critical problems both in theory and in practice, various  finite element methods (FEMs) approximating such singular solutions have been studied.  Intuitively, the accuracy of the numerical solution can be improved by increasing the mesh density near the singularity of the solution. For elliptic boundary value problems in two-dimensional (2D) polygonal domains, this idea has led to effective FEMs based on local mesh grading algorithms, in which the numerical approximation of singular solutions achieves the optimal convergence rate.  See \cite{Apel99, BKP79, BNZ205, BCS09,  LMN10,   MR0478669, Raugel78} and references therein.  The validation of these methods highly depends on the regularity estimate of the singular solution in special weighted Sobolev spaces (e.g., \cite{BNZ105, Dauge88, Grisvard85,MR2209521,Kondratiev67,  MR10,nicaise93}), which itself is an active research topic in mathematical analysis. 

For a three-dimensional (3D) polyhedral domain $\Omega$,  the solution is featured with different types of singularities:  the  vertex (conical) singularity and the (anisotropic) edge singularity. Thus, an anisotropic mesh is in general expected for a better finite element approximation. The combination of different types of singularities, together with  the complexity in the 3D  geometry, makes 
the development of optimal FEMs for equation (\ref{eqn.n1}) a more technically challenging task. Existing  algorithms on polyhedral domains usually require restrictive geometric conditions on the mesh and on the domain. Some relevant results are as follows. The mesh in \cite{MR1275108, MR2712527, LN195} is based on the method of dyadic partitioning. These meshes are  isotropic and optimal only for weaker singular solutions. The  mesh in  \cite{Apel99, AN98, ANS01, ASW96} is based on a coordinate transformation from a quasi-uniform mesh. It is anisotropic along the edges and requires confining angle conditions for the simplex. The  mesh in \cite{BLN12, BNZ107} is also anisotropic and leads to optimal convergence rate. The algorithm, however, requires extra steps for prism refinements to maintain the angle condition in the simplex. There are also tensor-product anisotropic meshes based on 2D  graded meshes \cite{AS02, MR3454358} that are usually effective on a domain with simple geometry.

The aim of this paper is twofold. First, we propose new anisotropic mesh refinement algorithms (Algorithm \ref{alg.n1}) for the finite element approximation of singular solutions in equation (\ref{eqn.n1}). These graded refinements are simple, explicit, and determined by a set of parameters associated to the singular set (vertices and edges) of the domain. Meanwhile, with less geometric requirements on the simplex and on the domain, these algorithms are  defined recursively based on direct decomposition of tetrahedra, and lead to conforming triangulations. Second, we develop $H^1$ interpolation error estimates  for the finite element space associated with the proposed anisotropic mesh. Due to the lack of regularity in the usual Sobolev space, these estimates are established for solutions in suitable weighted Sobolev spaces $\mathcal M^m_{\bmu}$ (Definition \ref{def.newweight}), in which the norm of the solution continuously depends on the given data $f$. Using the interpolation error estimates and weighted regularity results for the solution, one can decide the range of the grading parameters, such that the finite element solution approximates the singular solution in the optimal rate.  

A notable difference from the existing meshes is that our triangulation, with tetrahedral elements, in general violates the maximum angle condition  \cite{BA76}. Namely, the maximum interior angle in the triangular faces of the tetrahedra approaches $\pi$ as the level of refinements increases. To overcome the resulting difficulty in the error analysis, we develop technical tools through the following steps.   First, we classify tetrahedra into different types according to their relation with the singular set. For each tetrahedron type, we construct explicit linear transformations that map the tetrahedra to the reference element.  We show that the mapping is bounded, whose upper bound is independent of the refinement level. Then, we obtain the interpolation error estimate by proving that the lack of angle condition can be compensated for by different weights in the function space. The finite element error estimate (Theorem \ref{thm.mmm}) is an immediate consequence of the interpolation error analysis and the C\'ea Lemma.  The interpolation error estimate is independent of the regularity analysis for the solution.  The weighted space $\mathcal M^m_{\bmu}$ and some of its variants  are closely related to the Mellin transform  for non-smooth domains \cite{Kondratiev67, KMR01}, in which many rigorous regularity results have been established \cite{BCD03, MR2766898,MR1866274}. Thus,  using $\mathcal M^m_{\bmu}$ as the function space for the solution, to validate the proposed FEM, we here pay more attention on the connection between the grading parameters in the anisotropic mesh and the indices in the weighted space. Besides the results in this paper,  we  also expect that  the self-contained analytical techniques  developed here  will lead to new convergence results when similar weighted spaces are considered, and will be useful for other numerical studies of the the proposed anisotropic FEM.

The rest of the paper is organized as follows. In Section \ref{sec2}, we define the weighted Sobolev space and the finite element approximation to equation (\ref{eqn.n1}). In Section \ref{sec3}, we propose the 3D anisotropic mesh algorithm and discuss the resulting mesh properties. In Section \ref{sec4}, we give detailed interpolation error estimates on anisotropic meshes in weighed spaces. In Section \ref{sec5}, we report numerical test results on two model domains. These numerical results are in agreement with our theoretical prediction, and hence provide evidence for the validation of our method.

Throughout the text below, we adopt the bold notation for vector fields. Let $T$ be a triangle (resp. tetrahedron) with vertices $a, b, c$ (resp. $a, b, c, d$). Then, we denote $T$ by its vertices: $\triangle^3 abc$ for the triangle and  $\triangle^4abcd$ for the tetrahedron, where the sup-index implies the number of vertices for $T$. We denote by $ab$ the open line segment with endpoints $a$ and $b$ and denote by $\overrightarrow{ab}$ the vector from $a$ to $b$. By $a\sim b$ (resp. $a\lesssim b$), we mean there exists a constant $C>0$ independent of $a$ and $b$, such that $C^{-1}a\leq b\leq Ca$ (resp. $a\leq Cb$). In addition, when $A\subset B$, it is possible that $A=B$. The generic constant $C>0$ in our estimates  may be different at different occurrences. It will depend on the computational domain, but not on the functions  involved or the mesh level in the finite element algorithms.  

\section{Preliminaries} \label{sec2}
In this section, we introduce the weighted Sobolev space and the finite element approximation to equation (\ref{eqn.n1}), as well as other necessary notation and existing results. 
\subsection{Weighted Sobolev spaces}
Let $\mathcal V=\{v_\ell\}_{\ell=1}^{N_v}$ and  $\mathcal E=\{e_\ell\}_{\ell=1}^{N_e}$ be the set of vertices  and \textit{open} edges of $\Omega$, where  $N_v$ and $N_e$ are the numbers of the vertices and edges, respectively.    Let $N_s:=N_v+N_e$. Then, we denote  the singular set by $\mathcal S:=\{s_\ell\}_{\ell=1}^{N_s}=\maV\cup\maE$.  We number the elements in $\mathcal S$, such that  
\be\label{eqn.numbering}
s_\ell=v_\ell \quad {\rm{for}}\quad 1\leq \ell \leq N_v; \ s_\ell=e_{\ell-N_v}  \quad {\rm{for}}\quad N_v< \ell\leq N_s.
\ee
Namely, the first $N_s$ elements are vertices, while the last $N_e$ elements are edges.

Then, we classify different sub-regions based on their locations relative to the singular set $\mathcal S$.

\begin{definition}\label{def.domain} (The Domain Decomposition) For a vertex $v\in\mathcal V$,  let $\mathcal O_{v}\subset\Omega$  be a neighborhood of $v$, whose closure   does not contain any other vertices. Let $\mathcal G_v$ be the projection of $\mathcal O_v$ on the unit sphere $S^2$ centered   at $v$. Therefore, $\mathcal G_v$ is a polygon on $S^2$. Denote by $\mathcal E_v\subset\mathcal E$ the set of edges that touch $v$. Then, each edge $e\in\mathcal E_v$ corresponds to a vertex $v_e$ of the region $\mathcal G_v$.  Let $\mathcal O(v_e)\subset \mathcal G_v$ be a neighborhood of the vertex $v_e$, whose closure does not contain other vertices of $\mathcal G_v$. Then, using the spherical coordinates $(\rho, \vartheta)\in\mathbb R_+\times S^2$  centered at $v$, we define the neighborhood of the part of the edge $e\in\mathcal E_v$ close to $v$, $\mathcal O_e^v=\{(\rho, \vartheta)\in\mathcal O_v,\ \vartheta\in\mathcal O(v_e)\}$. Thus, $\Omega$ has the  decomposition
\begin{equation}\label{eqn.decomp}
\Omega=\Big(\cup_{v\in\mathcal V}\big(\mathcal O_v^o\cup(\cup_{e\in\mathcal E_v}\mathcal O^v_e)\big)\Big)\cup \Omega^o\cup(\cup_{e\in\mathcal E}\mathcal O_e^o), 
\end{equation}
where $\mathcal O^o_v=\mathcal O_v\setminus (\cup_{e\in\mathcal E_v}\mathcal O^v_e)$,  $\Omega^o$ is an interior region of $\Omega$ away from the singular points, and   $\mathcal O_e^o=\Omega\setminus\big(\Omega^o\cup(\cup_{v\in\mathcal V}\mathcal O_v)\big)$. Namely, $\Omega$ is decomposed into four components: (I) $\cup_{v\in\mathcal V}(\cup_{e\in\mathcal E_v}\mathcal O^v_e)$,  the neighborhood of the part of edges close to vertices;  (II) $\cup_{e\in\mathcal E}\mathcal O_e^o$, the neighborhood  of the part of edges away from vertices; (III) $\cup_{v\in\mathcal V}\mathcal O^o_v$, the sub-region of the neighborhood of the vertices that does not contain edge points; (IV) $\Omega^o$, the interior part away from  the singular set $\mathcal S$. \end{definition}

With this domain decomposition, we define the following weighted Sobolev space. 

\begin{definition}\label{def.newweight} (Anisotropic Weighted Spaces) Let $H^m$, $m\geq 0$, be the usual Sobolev space that consists of functions whose $k$th derivatives are square-integrable for $0\leq k\leq m$.  Let $H^{m}_{loc}(\Omega):=\{v, \ v\in H^m(\omega)\}$, where $\omega$ is   any open subset with compact closure $\bar\omega\subset\Omega$. Let $\rho_{v}(x)$ and $\rho_{e}(x)$ be distance functions from $x\in\Omega$ to the vertex  $v\in\maV$ and to the edge  $e\in\maE$, respectively. Within the neighborhood $\mathcal O^v_e$, let $\rho_{e, v}=\rho_e/\rho_v$ be the angular distance. 
 In the neighborhoods $\mathcal O^o_e$ and $\mathcal O^v_e$, we choose a local Cartesian coordinate system in which the edge $e\in\mathcal E$ lies on the $z$-axis. Let $\alpha_{\perp}=(\alpha_1, \alpha_2)$  consist of the first two entries of the multi-index $\alpha=(\alpha_1, \alpha_2, \alpha_3)\in \mathbb Z^3_{\geq 0}$. Therefore,  in $\mathcal O^o_e$ and $\mathcal O^v_e$, $\partial^{\alpha_\perp}=\pa_x^{\alpha_1}\pa_y^{\alpha_2}$ is a partial derivative in a direction perpendicular to the $z$-axis. Recall the singular set $\mathcal S$ in (\ref{eqn.numbering}). Then, given an $N_s$-dimensional vector $\bmu=(\mu_1, \cdots, \mu_{N_s})$, we define the anisotropic weighted space
\begin{eqnarray}\label{eqn.newweight}
\mathcal{M}^{m}_{\bmu}(\Omega):=\{v\in H^{m}_{loc}(\Omega),\ \rho_{v}^{|\alpha|-\mu_v}
    \partial^\alpha v \in L^2(\mathcal O^o_{v}), \  \rho_{e}^{|\alpha_\perp|-\mu_{e}}
    \partial^\alpha v \in L^2(\mathcal O^o_{e}),\\\ \rho_{v}^{|\alpha|-\mu_v}\rho_{e, v}^{|\alpha_\perp|-\mu_e}
    \partial^\alpha v \in L^2(\mathcal O^{v}_{e}),\ 
    \forall |\alpha|\leq m
    \}, \nonumber
\end{eqnarray}
where $\mu_v$ and $\mu_e$ are the entries in $\bmu$ that have the same sub-indices as those of $v$ and $e$ in $\mathcal S$. 
Thus, for $1\leq \ell\leq N_v$, $\mu_\ell$ specifies the weight associated to the vertex $v_\ell\in\mathcal V$;  and for $N_v<\ell\leq N_s$,  $\mu_\ell$ gives the weight associated to  the edges $e_{\ell-N_v}\in\mathcal E$.  For any $v\in\mathcal M^m_{\bmu}(\Omega)$, the associated norm is
\begin{eqnarray*}\|v\|^2_{\mathcal{M}^{m}_{\bmu}(\Omega)}:=\|v\|^2_{H^m(\Omega^o)}+\sum_{|\alpha|\leq m}\big(\sum_{v\in \maV}[\|\rho_{v}^{|\alpha|-\mu_v}
    \partial^\alpha v\|_{L^2(\mathcal O^o_{v})}^2+\sum_{\{e\in\mathcal E, \bar e\cap v=v\}}\|\rho_{v}^{|\alpha|-\mu_v}\rho_{e, v}^{|\alpha_\perp|-\mu_e}
    \partial^\alpha v\|^2_{L^2(\mathcal O^{v}_{e})}]\\
    +\sum_{e\in\maE}\| \rho_{e}^{|\alpha_\perp|-\mu_{e}}
    \partial^\alpha v\|_{L^2(\mathcal O^o_{e})}^2\big).
    \end{eqnarray*}
In this paper,  all the vectors denoted by the bold font  have $N_s$ entries. For any two  vectors $\ba$ and $\bb$, we write $\ba < (\leq, >, \geq)\ \bb$ if each entry $a_\ell< (\leq,>, \geq)\ b_\ell$, $1\leq \ell\leq N_s$. We denote by $\bone$ (resp. $\bzero$) the constant $N_s$-dimensional vectors with all entries being $1$ (resp. $0$). 
\end{definition}
Note that the distance functions in the space $\mathcal M^{m}_{\bmu}$ are determined by the  location of the singular set $\mathcal S$. Thus, they only depend on the domain $\Omega$, and remain the same for any sub-region of $\Omega$.
\begin{remark} 
The space $\mathcal M^{m}_{\bmu}$ is anisotropic in the sense that the transverse derivatives $\partial^{\alpha_\perp}$ and the longitudinal derivatives along the edge play different roles in the formulation. Compared with the isotropic weighted spaces \cite{BCD03}
\begin{eqnarray}\label{eqn.n2}
\mathcal{K}^{m}_{\bmu}(\Omega):=\{v\in H^{m}_{loc}(\Omega),\ \rho_{v}^{|\alpha|-\mu_v}
    \partial^\alpha v \in L^2(\mathcal O^o_{v}), \  \rho_{e}^{|\alpha|-\mu_{e}}
    \partial^\alpha v \in L^2(\mathcal O^o_{e}),\nonumber\\\ \rho_{v}^{|\alpha|-\mu_v}\rho_{e, v}^{|\alpha|-\mu_e}
    \partial^\alpha v \in L^2(\mathcal O^{v}_{e}),\ 
    \forall |\alpha|\leq m
    \}, \nonumber
\end{eqnarray}
the space $\mathcal M^{m}_{\bmu}$ is suitable to describe the anisotropic behavior of singular solutions, especially the additional regularity in the edge direction. For example, define the vector $\bbeta$, such that
\begin{equation}\label{eqn.eta}
\eta_\ell=\sqrt{\lambda_\ell+1/4} \ \ \text{for}\ 1\leq\ell\leq N_v\ \ \text{and}\ \ \eta_{\ell}={\pi}/{\omega_{\ell}}\ \  \text{for}\ N_v<\ell\leq N_s,
\end{equation}
where  $\lambda_\ell>0$ is the smallest positive eigenvalue of the Laplace-Beltrami operator with the zero Dirichlet boundary condition on the polygon $\mathcal G_{v_\ell}$ in the unit sphere $S^2$ centered at $v_\ell$, and $\omega_\ell$ is the interior dihedral angle of the edge $e_{\ell-N_v}\in\maE$. Then,   for $m\geq 0$, the solution $u\in H^1_0(\Omega)$ of equation \eqref{eqn.n1},  satisfies  \cite{BCD03, CDN14} 
\begin{eqnarray}\label{eqn.regani}
\|u\|_{\mathcal M^{m}_{\ba+\bone}(\Omega)}\leq C\|f\|_{\mathcal M^{m}_{\ba-\bone}(\Omega)},\ \qquad {\rm{for}}\ \bzero\leq \ba< \bbeta.
\end{eqnarray}
This shows the continuous dependence of the solution on the given data in weighted spaces, despite the lack of regularity in usual Sobolev spaces.
\end{remark}
\begin{remark}  Note that the estimate (\ref{eqn.regani}) does not give a shifting in the index $m$. In fact, a smoother $f$ is expected  in order for $u$ to be in $\mathcal M^{m+2}_{\ba+\bone}(\Omega)$ \cite{BNZ107}. This, however, requires sophisticated regularity analysis that we will address in a forthcoming paper. Nevertheless, our  goal in this paper is to propose new anisotropic finite element algorithms and develop  interpolation error estimates in suitable weighted spaces. These estimates will also facilitate the finite element analysis for singular solutions when other anisotropic regularity results become available. Hence, from now on, we assume the solution of equation (\ref{eqn.n1}) satisfies 
$$u\in \mathcal M^{m+1}_{\bsigma+\bone}(\Omega) \qquad {\rm{for}}\ m\geq 1,$$ where  $\bsigma>\bzero$ will be specified later. 
\end{remark}

\subsection{The finite element method}
Recall that $H^1_0(\Omega)\subset H^1(\Omega)$ is the subspace consisting of functions with  zero trace on $\pa\Omega$. 
The variational solution $u\in H^1_0(\Omega)$ of equation (\ref{eqn.n1}) satisfies
\begin{equation*}
a(u, v)=\int_{\Omega}\nabla u\cdot \nabla v dx=\int_{\Omega}fvdx=(f, v),\quad \forall v\in H^1_0(\Omega).
\end{equation*}
Let $\maT_n$ be a triangulation of $\Omega$ with tetrahedra. Let $S_n\subset H^1_0(\Omega)$ be the Lagrange finite element space of degree $m\geq 1$ associated with  $\maT_n$. Namely, $S_n=\{v\in C(\Omega),\ v|_{T}\in \mathbb P_m, \  {\rm{for\ any\ tetrahedron}}\ T\in \maT_n\}$, where $\mathbb P_m$ is the space of polynomials of degree $\leq m$. Then, the finite element solution $u_n\in S_n$ for equation (\ref{eqn.n1}) is defined by 
\begin{equation}\label{eqn.fems1}
a(u_n, v_n)=(f, v_n),\quad \forall v_n\in S_n.
\end{equation}
\begin{remark} By the Poincar\'e inequality, the bilinear form $a(\cdot, \cdot)$ is both continuous and coercive on $H^1_0(\Omega)$. Thus, the C\'ea Lemma \cite{BS02,Ciarlet78} gives rise to 
\be\label{eqn.subo111}
\|u-u_n\|_{H^1(\Omega)}\leq \inf_{v_n\in S_n}\|u-v_n\|_{H^1(\Omega)}.
\ee
On a standard quasi-uniform triangulation $\maT_n$, it is well known that the limited regularity of  $u$  in the usual Sobolev space may result in a \textit{sub-optimal} convergence rate  for the finite element approximation. Namely, 
\be\label{eqn.subo}
\|u-u_n\|_{H^1(\Omega)}\leq Ch^{s}\|u\|_{H^{s+1}(\Omega)}, 
\ee
where $h$ is the mesh size in $\maT_n$ and  $0<s<m$ depends on the geometry of the domain. 
\end{remark}

\section{Anisotropic mesh algorithms} \label{sec3}
In this section, we propose new 3D anisotropic mesh algorithms for the finite element approximation of singular solutions of equation (\ref{eqn.n1}). We first classify tetrahedra in the triangulation based on their relation with the singular set $\mathcal S$.

\begin{definition}\label{def.n1} (Tetrahedron Types) Recall the vertex set $\mathcal V$, the edge set $\mathcal E$, and $\mathcal S=\maV\cup\maE$. For a tetrahedron $T$,  we say $T$ contains a \textit{singular edge} if one of its edges lies on an edge in $\maE$. Let $x$ be a vertex of $T$.  We say $x$ is   a \textit{singular vertex} of $T$ if $x\in \maV$, or $x\in e$ for some edge  $e\in\maE$ but none of $T$'s edges lies on $e$. Let $\mathcal T$ be an initial triangulation of $\Omega$, such that (I) each tetrahedron contains at most one singular vertex and at most one singular edge; (II) if a tetrahedron contains both a singular vertex and a singular edge, the singular vertex  is an endpoint of the singular edge.  Then, each tetrahedron $T\in \mathcal T$ falls into one of the five categories. 
\begin{itemize}
\item[1.] $o$-tetrahedron: $\bar T \cap \mathcal S=\emptyset$.
\item[2.]  $v$-tetrahedron:  $\bar T\cap \mathcal S$ is a vertex in $\maV$.
\item[3.] $v_e$-tetrahedron: $\bar T\cap \mathcal S$ is an interior point of an edge in $\maE$.
\item[4.]  $e$-tetrahedron:   $\bar T \cap \mathcal S$ is an edge of $T$, which lies on an edge in $\mathcal E$ but contains no vertex  in $\mathcal V$.
\item[5.]  $ev$-tetrahedron: $\bar T\cap \mathcal S$ contains  a vertex  $v\in\mathcal V$ and  an edge of $T$ that lies on an  edge in $\mathcal E$ joining $v$.
\end{itemize}
\end{definition}

Then, we present our  anisotropic mesh algorithm.

\begin{figure}
\includegraphics[scale=0.57]{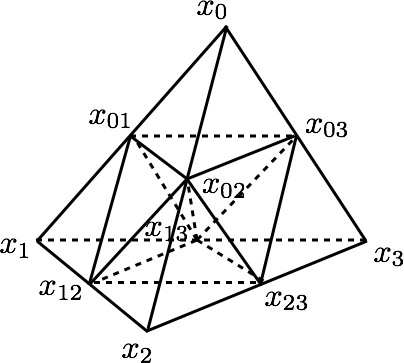}\hspace{0.2cm}
\includegraphics[scale=0.27]{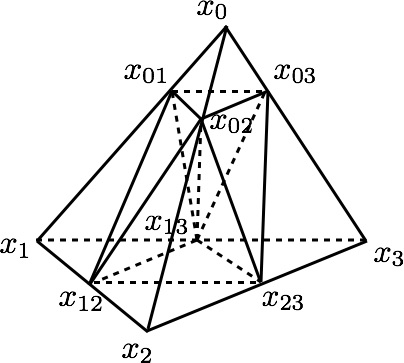}\hspace{.2cm}
\includegraphics[scale=0.27]{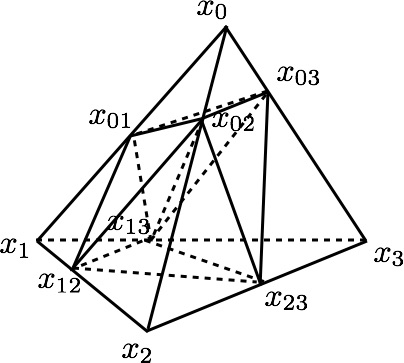}\hspace{.2cm}\\
\includegraphics[scale=0.27]{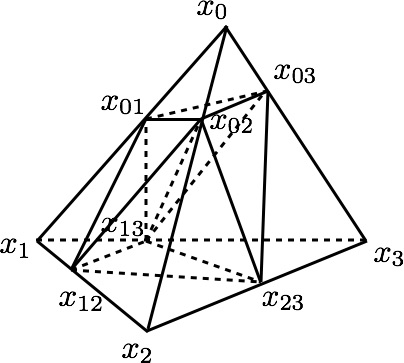}\hspace{.6cm}
\includegraphics[scale=0.27]{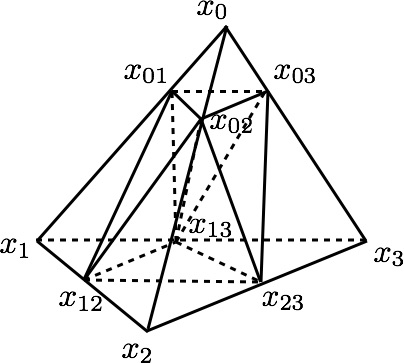}
\caption{{Decompositions for a tetrahedron $\triangle^4x_0x_1x_2x_3$,  top row (left -- right) :  $o$-tetrahedron,   $v$- or $v_e$-tetrahedron,  $e$-tetrahedron;  bottom row  (left -- right):  an $ev$-tetrahedron  ($\kappa_{ev}=\kappa_e$), an $ev$-tetrahedron ($\kappa_{ev}=\kappa_v$).}}\label{fig.n3}
\end{figure}

\begin{algorithm}\label{alg.n1} (Anisotropic Refinement)  Let $\mathcal T$ be a triangulation of $\Omega$ as in Definition \ref{def.n1}. For each  element $s_\ell\in\mathcal S$, $1\leq \ell\leq N_s$, we associate a grading parameter $\kappa_\ell\in(0, 1/2]$. Let $T=\triangle^4x_0x_1x_2x_3\in \mathcal T$ be a tetrahedron  with  vertices $x_0, x_1, x_2$, $x_3$, such that $x_0$ is  the singular vertex if $T$ is a $v$-, $v_e$-, or $ev$-tetrahedron; and $x_0x_1$ is the singular edge if $T$ is an $e$- or $ev$-tetrahedron. Let $\bkappa=(\kappa_1, \cdots, \kappa_{N_s})$ be the collection of the grading parameters, such that each $\kappa_\ell$ corresponds to an element $s_\ell$ in the singular set $\mathcal S$. Then, 
 the refinement,  denoted by $\bkappa(\mathcal T)$,  proceeds as follows.  We first
    generate new nodes $x_{kl}$, $0\leq k<l\leq 3$, on each edge $x_kx_l$ of
    $T$, based on its type.
    \begin{itemize}
 \item[(I)] ($T$ is an $o$-tetrahedron.):  $x_{kl}=(x_k+x_l)/2$.
  \item[(II)] ($T$ is a $v$-tetrahedron.): Suppose $x_0=s_\ell\in\maV$ ($1\leq \ell\leq N_v$). Define $\kappa_v:=\kappa_\ell$. Let $I_v:=\{\ell,\ v{\rm{\ is\ an\ endpoint\ of}}\ e_{\ell-N_v}\}$ be the index set for edges touching $v$. Define $\kappa=\kappa_{ev}:=\min_{\ell\in I_v}(\kappa_v, \kappa_\ell)$. Then,  $x_{kl}=(x_k+x_l)/2$ for $1\leq k< l\leq 3$;  $x_{0l}=(1-\kappa)x_0+\kappa x_l$ for $1\leq l\leq 3$.
  \item[(III)] ($T$ is a $v_e$-tetrahedron.): Suppose $x_0=\overline {x_0x_1}\cap s_\ell$, $(N_v< \ell\leq N_s)$, namely, $s_\ell\in\maE$.  We define $\kappa=\kappa_e:=\kappa_\ell$. Then,  $x_{kl}=(x_k+x_l)/2$ for $1\leq k< l\leq 3$;  $x_{0l}=(1-\kappa)x_0+\kappa x_l$ for $1\leq l\leq 3$.

  \item[(IV)] ($T$ is an $e$-tetrahedron.): Suppose $x_0x_1\subset e_{\ell-N_v}=s_\ell\in\mathcal E$ ($N_v< \ell\leq N_s$). Define $\kappa_e:=\kappa_\ell$.  Then, $x_{kl}=(1-\kappa_{e})x_k+\kappa_{e} x_l$ for $0\leq k\leq 1$ and $2\leq l\leq 3$; $x_{01}=(x_0+x_1)/2$, $x_{23}=(x_2+x_3)/2$.
  \item[(V)] ($T$ is an $ev$-tetrahedron.): Suppose $x_0=v_\ell=s_\ell \in \mathcal V$ ($1\leq \ell\leq N_v$) and $x_0x_1\subset e_{\ell'-N_v}=s_{\ell'}\in\mathcal E$ ($N_v< \ell'\leq N_s$). Define $\kappa_v:=\kappa_\ell$, $\kappa_e:=\kappa_{\ell'}$, and  $\kappa_{ev}:=\min_{\ell\in I_v}(\kappa_v, \kappa_\ell)$, where $I_v$ is the index set defined  in (II). Then, for $2\leq l\leq 3$, $x_{0l}=(1-\kappa_{ev})x_{0}+\kappa_{ev} x_l$ and $x_{1l}=(1-\kappa_{e})x_{1}+\kappa_{e} x_l$; $x_{01}=(1-\kappa_v)x_0+\kappa_v x_1$, $x_{23}=(x_2+x_3)/2$. 
  \end{itemize}
Connecting
    these nodes $x_{kl}$ on all the faces of $T$, we obtain four sub-tetrahedra
    and one octahedron. The octahedron then is cut into four
    tetrahedra using $x_{13}$ as the common vertex. Therefore, after
    one refinement, we obtain eight sub-tetrahedra  for each $T\in\mathcal T$ denoted by their vertices:
\begin{eqnarray*}
  &\triangle^4x_0x_{01} x_{02} x_{03}, \ \triangle^4x_1 x_{01} x_{12}
  x_{13},\ \triangle^4x_2 x_{02} x_{12} x_{23},\ \triangle^4x_3 x_{03} x_{13}
  x_{23},\\
  &\triangle^4x_{01} x_{02} x_{03} x_{13}, \ \triangle^4x_{01} x_{02} x_{12}
  x_{13},\ \triangle^4x_{02} x_{03} x_{13} x_{23},\ \triangle^4x_{02} x_{12}
  x_{13} x_{23}.
\end{eqnarray*}
See Figure \ref{fig.n3} for different types of decompositions.  Given an initial mesh $\mathcal T_0$ satisfying the condition in Definition \ref{def.n1}, the associated family of anisotropic meshes $\{\mathcal T_n,\ n\geq0\}$ is defined recursively $\mathcal T_n=\bkappa(\mathcal T_{n-1})$. See Figure \ref{fig.n4} for example.
\end{algorithm}

\begin{figure}
\includegraphics[scale=0.10]{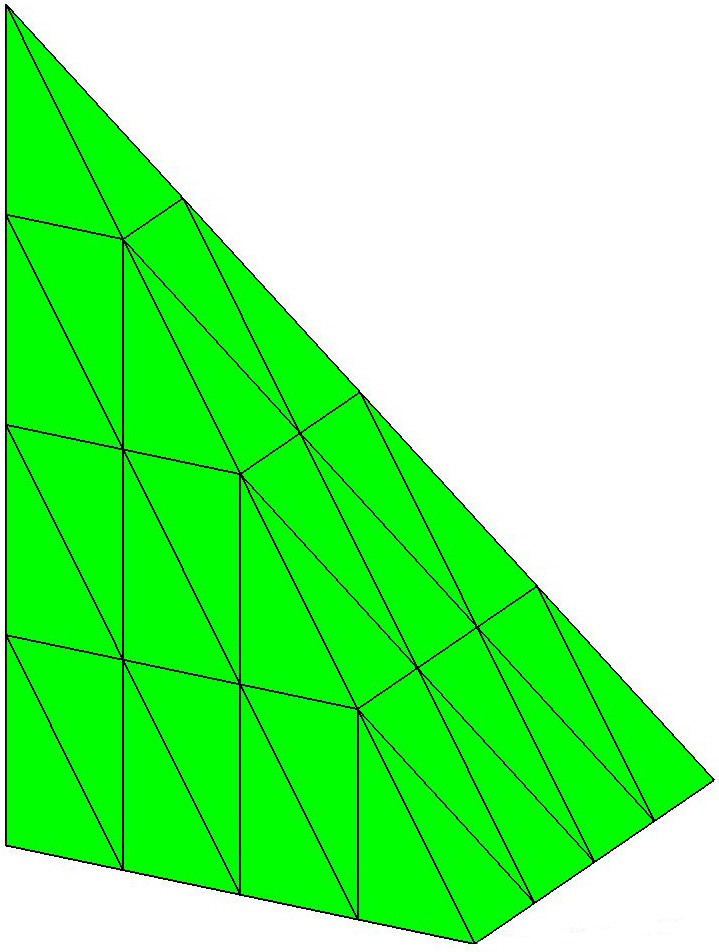} \hspace{1.5cm}
\includegraphics[scale=0.10]{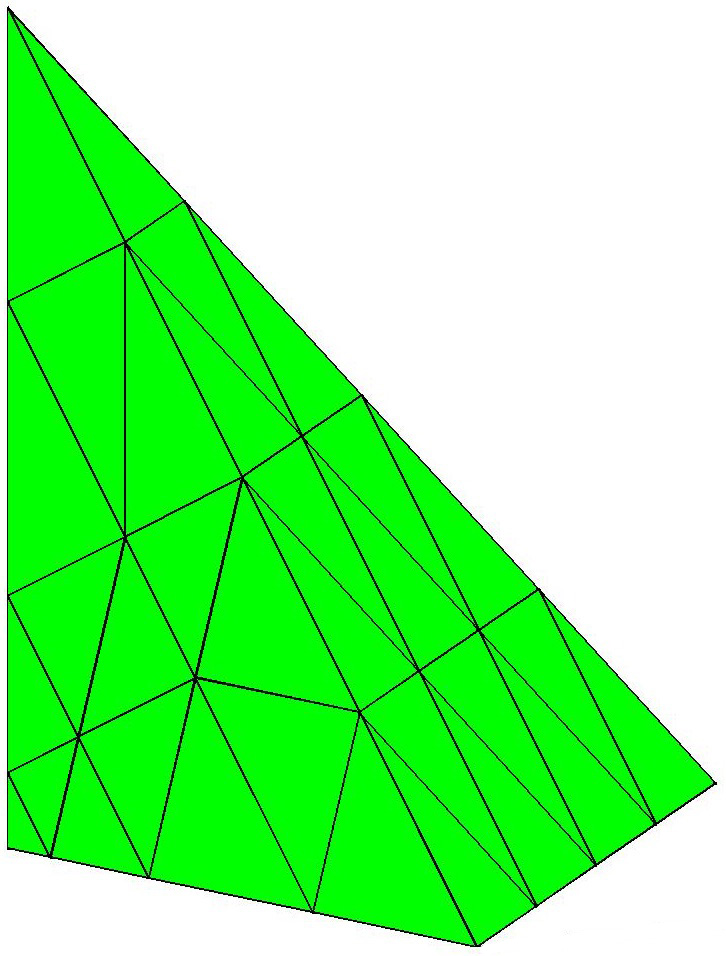}\hspace{1.5cm}
\includegraphics[scale=0.10]{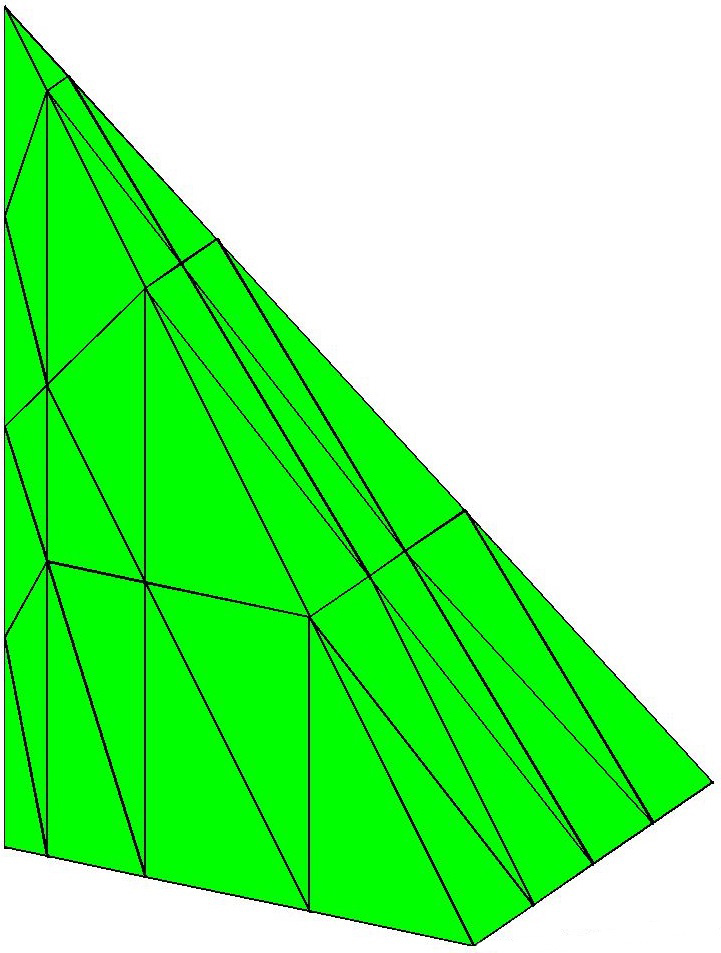}\\
\includegraphics[scale=0.10]{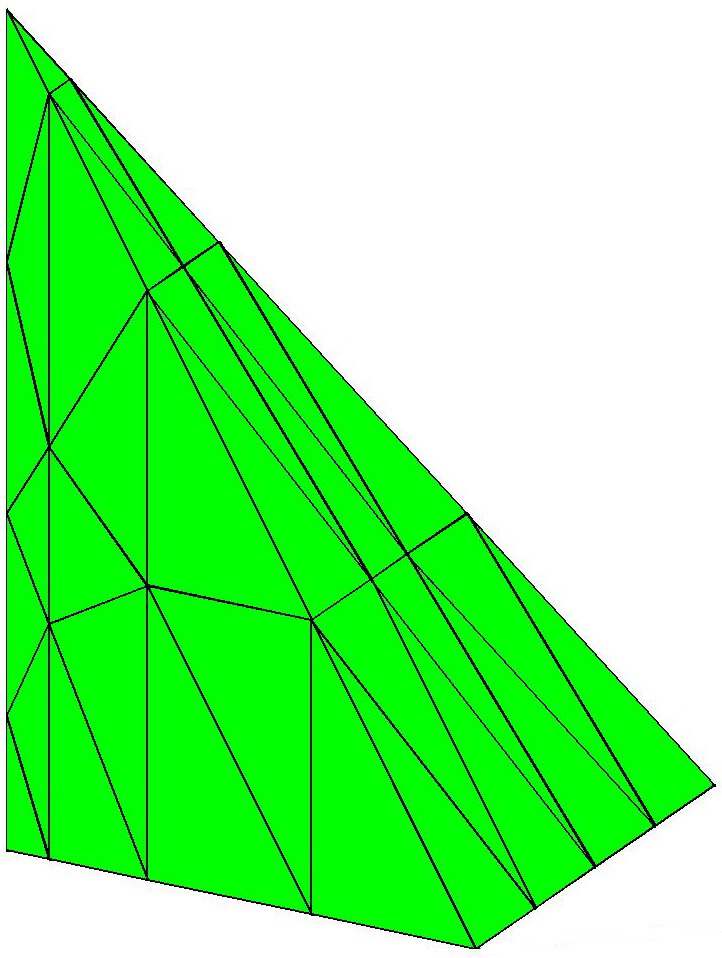}\hspace{2cm}
\includegraphics[scale=0.10]{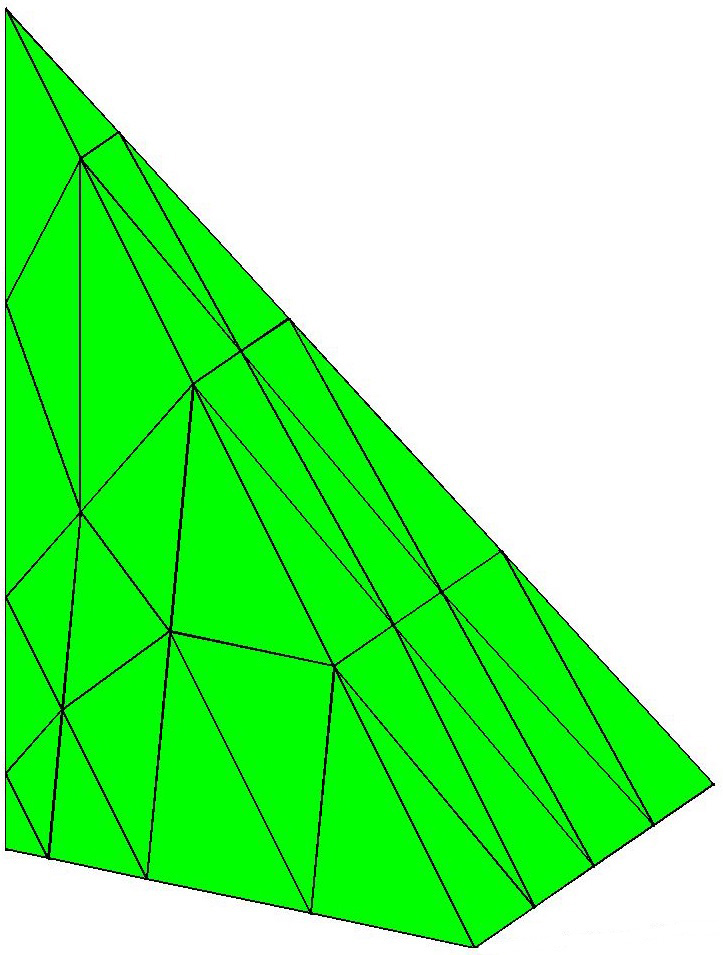}
\caption{Anisotropic triangulations  after two consecutive refinements on  a tetrahedron, top row  (left -- right): $o$-tetrahedron, $v$- or $v_e$-tetrahedron  ($\kappa=0.3$), $e$-tetrahedron ($\kappa_{e}=0.3$); bottom row (left -- right):  $ev$-tetrahedron ($\kappa_{ev}=0.3, \kappa_v=0.4, \kappa_{e}=0.3$), $ev$-tetrahedron ($\kappa_{ev}=0.3, \kappa_{v}=0.3, \kappa_e=0.4$).}\label{fig.n4}
\end{figure}

\begin{remark} Algorithm \ref{alg.n1} first assigns to each singular element $s_\ell\in\mathcal S$ a grading parameter $\kappa_\ell$, which can be regarded as an indicator of the severity of the singularity at $s_\ell$. A smaller value of $\kappa_\ell$ leads to a higher mesh density  near $s_\ell$, while the value $\kappa_\ell=1/2$ corresponds to a quasi-uniform refinement.  It is apparent that our meshing method results in very different mesh geometries. In a region away from the singular set $\mathcal S$ (i.e., $\Omega^o$), the mesh is isotropic and quasi-uniform. The local refinement for a $v$- or $v_e$-tetrahedron in fact follows the same rule: the mesh is isotropic and graded toward the vertex $x_0$ based on the  grading parameter $\kappa$ associated to the vertex $x_0$. In the neighborhood $\maO_e^o$ of an edge away from the vertices, the resulting mesh in general is anisotropic and graded toward the  edge $e\in\maE$. The mesh refinement in $\maO_e^v$ depends on the parameters $\kappa_v$ and $\kappa_\ell$, $\ell\in I_v$, which is also anisotropic, graded toward both the  edge $e\in\maE$ and the vertex $v\in\maV$.
\end{remark}

\begin{remark}
Our anisotropic refinements also generate tetrahedra with different shape regularities. A direct calculation shows that successive  refinements of an $o$-tetrahedron produce tetrahedra within  three similarity classes \cite{Bey95};  refinements for a $v$- or $v_e$-tetrahedron  produce tetrahedra within  22 similarity classes (Remark 3.4 in \cite{HLNU}). However,  refinements for an $e$- or $ev$-tetrahedron   lead to anisotropic meshes toward the edge that in general do not preserve the maximum angle condition. Namely, the maximum edge angle in the face of the tetrahedron approaches $\pi$ as the level of refinement $n$ increases. This is  a main difficulty that we shall overcome in the error analysis.
 \end{remark}


\begin{remark}
 Compared with existing 3D graded mesh refinements \cite{Apel99, AN98, BNZ107, MR3454358}, the proposed algorithm  has a few notable properties: 1. it is simple, explicit, and defined recursively; 2. the meshes $\mathcal T_j$, $j\leq n$, are conforming and the associated finite element spaces $S_j$ are nested; 3. the algorithm  results in a triangulation with the same topology and data structure as the usual 3D uniform mesh \cite{Bey95}, and also provides  the flexibility to adjust the grading parameters for vertex and edge singularities on general polyhedral domains. In what follows, we shall obtain interpolation error estimates for singular solutions on such meshes, which in turn imply that our mesh   can effectively improve the convergence of the finite element approximation.\end{remark}

To simplify the exposition in the next section, for a singular element in $\mathcal S$, we now define another set of parameters associated with $\kappa_v, \kappa_e$, and $\kappa_{ev}$ (Algorithm \ref{alg.n1}). When  $s_\ell\in\mathcal S$ is a vertex $v$, let $a_v$ be such that 
\be
\kappa_v=\kappa_\ell=2^{-m/a_v}.\label{eqn.kv1}
\ee
When $s_\ell\in\mathcal S$ is an edge $e$, let $a_e$ be such that
\be
\kappa_e=\kappa_\ell=2^{-m/a_e}\label{eqn.ke1}.
\ee
In addition,  we define the constant  $a_{ev}=\min_{\ell\in I_v}(a_v, a_\ell)$, where $I_v$ is the index set defined in (II) of Algorithm \ref{alg.n1}. Therefore, 
\be\label{eqn.kev1}
\kappa_{ev}=2^{-m/a_{ev}}.
\ee
Then, we denote by $\ba=(a_1, \cdots, a_{N_s})$ the collection of these mesh parameters 
\be\label{eqn.aell}
a_\ell:=\left\{\begin{array}{ll}
a_{ev}, \qquad \qquad & {\rm{if}}\ v=s_\ell;\\
a_e, \qquad \qquad & {\rm{if}}\ e=s_\ell,
\end{array}\right. \qquad 1\leq \ell\leq N_s.
\ee
Here, $m\geq 1$ is the polynomial degree in the finite element approximation (\ref{eqn.fems1}). Since $\kappa_\ell\in(0, 1/2]$, it is clear that $0<a_\ell\leq m$.

\section{Interpolation error estimates}\label{sec4}
In this section, we develop analytical tools and obtain interpolation error estimates on the proposed anisotropic mesh.  
Let $\maT_0$ be an initial triangulation of the domain $\Omega$ with tetrahedra  that satisfy the condition in Definition \ref{def.n1}. Recall $\maT_n$ is the mesh obtained after $n$  successive refinements based on the parameter $\bkappa$. 
Throughout this section, we let $h:=2^{-n}$ be the mesh parameter of $\maT_n$. For a continuous function $v$, we let $v_{I}$ be its Lagrange nodal interpolation associated to the underlying mesh.

Note that the tetrahedra in the initial mesh $\maT_0=\{T_{(0),j}\}_{j=1}^J$ are all shape regular and can be classified into five categories (Definition \ref{def.n1}). Thus, with the triangulation $\maT_n$, the  interpolation error estimates on $\Omega$  break down  into the interpolation error estimates on the sub-regions of $\Omega$, each of which is represented by an initial tetrahedron  $T_{(0),j}\in\maT_0$. 

In addition, we mention that  based on the definition,  the space $\mathcal M_{\bmu}^{m+1}$, $m\geq 1$, regardless of the sub-index $\bmu$, is equivalent to the Sobolev space $H^{m+1}$ on any sub-region of $\Omega$ that is away from the singular set $\mathcal S$. Therefore, by the Sobolev embedding Theorem, $u\in\mathcal M_{\bmu}^{m+1}(\Omega)$ is continuous at each nodal point in the interior of the domain. On the boundary of the domain, we set $u_I=0$ due to the boundary condition. This makes the interpolation $u_I$ well defined.

\subsection{Estimates on initial $o$-, $v$-, and $v_e$-tetrahedra in $\maT_0$} We first have the estimate for an $o$-tetrahedron in the initial mesh.
\begin{lemma} \label{lem.ote}
Let $T_{(0)}\in\maT_0$ be an $o$-tetrahedron. For $u\in\mathcal M^{m+1}_{\ba+\bone}(\Omega)$, where $\ba$ is given in (\ref{eqn.aell}), let $u_I$ be its nodal interpolation on $\maT_n$. Then, we have
\begin{eqnarray}
|u-u_I|_{H^1(T_{(0)})}\leq Ch^{m}\|u\|_{\mathcal M^{m+1}_{\ba+\bone}(T_{(0)})},
\end{eqnarray}
where $h=2^{-n}$ and $C$ is independent of $n$ and $u$.
\end{lemma}
\begin{proof} Based on Algorithm \ref{alg.n1},  the restriction of $\maT_n$ on $T_{(0)}$ is a  quasi-uniform mesh with  size $O(2^{-n})$. Since  $H^{m+1}(T_{(0)})$ is equivalent to $\mathcal M^{m+1}_{\ba+\bone}(T_{(0)})$ on an $o$-tetrahedron, by the standard interpolation error estimate, we have 
$$
|u-u_I|_{H^1(T_{(0)})}\leq C2^{-nm}\|u\|_{H^{m+1}(T_{(0)})}\leq Ch^{m}\|u\|_{\mathcal M^{m+1}_{\ba+\bone}(T_{(0)})}.
$$
This completes the proof.
\end{proof}

For a $v$- or $v_e$-tetrahedron in $\maT_0$, we first  identify its sub-regions that  have comparable distances to the singular vertex. 
\begin{definition} \label{def.vlayer}(Mesh Layers in $v$- and $v_e$-tetrahedra) Let $T_{(0)}=\triangle^4x_0x_1x_2x_3\in \maT_0$ be either a $v$- or a $v_e$-tetrahedron with $x_0\in\maV$ or $x_0\in e\in \maE$. We use a local Cartesian coordinate system, such that $x_{0}$ is the origin.  For $1\leq i\leq n$, the $i$th refinement on $T_{(0)}$ produces a small tetrahedron with $x_{0}$ as a vertex and with one face, denoted by $P_{v,i}$, parallel to the face $\triangle^3x_1x_2x_3$ of $T_{(0)}$. See Figure \ref{fig.n3} for example. 

Then, after $n$ refinements, we define the $i$th mesh layer $L_{v, i}$ of  $T_{(0)}$, $1\leq i<n$, as the region in $T_{(0)}$ between $P_{v,i}$ and $P_{v, i+1}$. We denote by $L_{v, 0}$ the region in  $T_{(0)}$ between $\triangle^3x_1x_2x_3$ and $P_{v, 1}$; and let $L_{v, n}$ be the small tetrahedron with $x_0$ as a vertex that is  bounded by $P_{v, n}$ and three faces of $T_{(0)}$. Since  it is clear that $x_0$ is the only point for the special refinement, we drop the sub-index $\ell$ in the grading parameter   (\ref{eqn.aell}). Namely, for such $T_{(0)}$, we use
$$\kappa=2^{-m/a}$$
to denote the grading parameter near $x_0$ ($\kappa=\kappa_{ev}$ if $x_0\in\maV$ and $\kappa=\kappa_e$ if $x_0\in e\in\mathcal E$).  
Define the dilation matrix
\begin{eqnarray}\label{eqn.bv1}
\mathbf B_{v,i}:= \begin{pmatrix}
   \kappa^{-i}   &   0 & 0 \\
  0    &   \kappa^{-i} & 0\\
  0 & 0 & \kappa^{-i}
\end{pmatrix}.
\end{eqnarray}
Then, by Algorithm \ref{alg.n1}, $\mathbf B_{v,i}$ maps $L_{v, i}$ to $L_{v, 0}$ for $0\leq i< n$, and maps $L_{v, n}$ to $T_{(0)}$. We define the \textit{initial triangulation} of $L_{v, i}$, $0\leq i<n$,  to be the first decomposition of $L_{v, i}$ into tetrahedra. Thus, the initial triangulation of $L_{v, i}$ consists of those tetrahedra in $\maT_{i+1}$ that are contained in the layer $L_{v, i}$.
\end{definition}

\begin{remark}\label{rk.56}
Based on the refinement,   on $L_{v, i}$, $0\leq i\leq n$, the tetrahedra  in $\maT_n$ are isotropic with mesh size $O(\kappa^i2^{i-n})$. In $T_{(0)}$, let $\rho$ be the distance to $x_0$. Therefore, 
\be\label{eqn.dk}
\rho\sim \kappa^i \qquad {\rm{on}}\ L_{v,i}, \quad 0\leq i<n.
\ee 
Namely, if $T_{(0)}$ is a $v$-tetrahedron, $\rho\sim\rho_v$ for $v=x_0\in\maV$; and if $T_{(0)}$  is a $v_e$-tetrahedron, $\rho\sim\rho_e$, where $e\in\maE$ is the edge containing $x_0$. \end{remark}

Then, we have the interpolation error estimate in the layer $L_{v, i}$.
\begin{lemma}\label{lem.vv} Let $T_{(0)}\in\maT_0$ be either a $v$- or a $v_e$-tetrahedron. For $u\in\mathcal M^{m+1}_{\ba+\bone}(\Omega)$, where $\ba$ is given in (\ref{eqn.aell}), let $u_I$ be its nodal interpolation on $\maT_n$.  Then, for $0\leq i<n$, we have 
\begin{eqnarray*}
|u-u_I|_{H^1(L_{v,i})}\leq Ch^{m}\|u\|_{\mathcal M^{m+1}_{\ba+\bone}(L_{v,i})},
\end{eqnarray*}
where $h=2^{-n}$ and $C$ is independent of $n$ and $u$.
\end{lemma}
\begin{proof}
For $(x,y , z)\in L_{v,i}$, let $(\hat x, \hat y, \hat z)\in L_{v,0}$ be its image under the dilation  $\mathbf B_{v,i}$. For a function $v$ on $L_{v,i}$, we define $\hat v$ on $L_{v,0}$ by
$$
\hat v(\hat x, \hat y, \hat z):=v(x, y, z).
$$
As part of $\maT_n$, the triangulation on $L_{v,i}$ is mapped by $\mathbf B_{v,i}$ to a triangulation on $L_{v,0}$ with mesh size $O(2^{i-n})$.   
 Then, by the scaling argument and (\ref{eqn.dk}), we have
\ben
|u-u_I|_{H^1(L_{v,i})}^2&=&\kappa^i |\hat u-\hat u_I|_{H^1(L_{v,0})}^2\leq C\kappa^i2^{2m(i-n)}|\hat u|^2_{H^{m+1}(L_{v,0})}\\
&\leq& C2^{2m(i-n)}\kappa^{2mi}|u|^2_{H^{m+1}(L_{v, i})}\leq C2^{2m(i-n)}\kappa^{2ai}\sum_{|\alpha|=m+1}\|\rho^{m-a}\pa^\alpha u\|^2_{L^2(L_{v, i})}.
\een
Recall $\kappa=2^{-m/a}$ and Remark \ref{rk.56}. 
Then, by the definition of the weighted space, we have
$$
|u-u_I|^2_{H^1(L_{v,i})}\leq C2^{2m(i-n)}\kappa^{2ai}\sum_{|\alpha|=m+1}\|\rho^{m-a}\pa^\alpha u\|^2_{L^2(L_{v, i})}\leq Ch^{2m}\|u\|^2_{\mathcal M^{m+1}_{\ba+\bone}(L_{v,i})},
$$
which completes the proof.
\end{proof}

Then, we give the error estimate on the entire initial tetrahedron $T_{(0)}$. 
\begin{corollary}\label{cor.vte}
Let $T_{(0)}\in\maT_0$ be either a $v$- or a $v_e$-tetrahedron. For $u\in\mathcal M^{m+1}_{\ba+\bone}(\Omega)$, where $\ba$ is given in (\ref{eqn.aell}), let $u_I$ be its nodal interpolation on $\maT_n$. Then, we have 
\begin{eqnarray*}
|u-u_I|_{H^1(T_{(0)})}\leq Ch^{m}\|u\|_{\mathcal M^{m+1}_{\ba+\bone}(T_{(0)})},
\end{eqnarray*}
where $h=2^{-n}$ and $C$ is independent of $n$ and $u$.
\end{corollary}
\begin{proof}
By Lemma \ref{lem.vv}, it suffices to show the estimate for the last layer $L_{v,n}$. For $(x,y , z)\in L_{v,n}$, let $(\hat x, \hat y, \hat z)\in T_{(0)}$ be its image under the dilation $\mathbf B_{v,n}$. For a function $v$ on $L_{v,n}$, we define $\hat v$ on $T_{(0)}$ by
$$
\hat v(\hat x, \hat y, \hat z):=v(x, y, z).
$$
Now let $\chi$ be a smooth cutoff function on $T_{(0)}$ such that
$\chi=0$ in a neighborhood of $x_0$ and $=1$ at every other node of
$T_{(0)}$.   Recall the distance function $\rho$ from Remark \ref{rk.56}. Thus, $\rho(\hat x, \hat y, \hat z)=\kappa^{-n}\rho(x, y, z)$.   Since $\chi\hat u=0$ in the neighborhood of $x_0$,  we have
$$|\chi\hat u|_{H^{m+1}(T_{(0)})}^2\leq C\sum_{|\alpha|\leq m+1}\| \rho^{|\alpha|-1}\pa^{\alpha}\hat u\|_{L^2(T_{(0)})}^2.$$ 
Define $\hat w:=\hat u-\chi\hat u$ and note that $(\chi\hat u)_{I}=\hat u_I$.  We have
\begin{eqnarray}
    |\hat u - \hat u_{I} |_{H^1
    (T_{(0)})}& =&  |\hat w+\chi \hat u - \hat u_{I} |_{H^1 (T_{(0)})}  \leq  |\hat w|_{H^1
    (T_{(0)})}+|\chi\hat u-\hat u_{I}|_{H^1
    (T_{(0)})}\nonumber\\
  &=& |\hat w|_{H^1 (T_{(0)})}+|\chi\hat u -
  (\chi\hat u)_{I}|_{H^1 (T_{(0)})}\leq C(\|\hat u\|_{H^{1}(T_{(0)})}+|\chi\hat u|_{H^{m+1}(T_{(0)})}), \label{eqn.new1111}
\end{eqnarray}
where $C$ depends on $m$ and, through $\chi$, the nodes in the
triangulation. Then, using \eqref{eqn.new1111}, the scaling argument, $\kappa^{-n}\lesssim \rho^{-1}$ in $L_{v, n}$, the definition of the weighted space, and (\ref{eqn.aell}), we have
\begin{eqnarray*}
    |u - u_{I}|^2_{H^1 (L_{v,n})} & =& 
    \kappa^n|\hat u-\hat u_I|_{H^1(T_{(0)})}^2\leq C \kappa^n(\|\hat u\|_{H^{1}(T_{(0)})}^2+\sum_{|\alpha|\leq m+1}\| \rho^{|\alpha|-1}\pa^{\alpha}\hat u\|_{L^2(T_{(0)})}^2)\\
    & \leq& C\sum_{|\alpha|\leq m+1}\| \rho^{|\alpha|-1}\pa^{\alpha} u\|_{L^2(L_{v,n})}^2\leq C\kappa^{2na} \|u\|^2_{\mathcal M_{\ba+\bone}^{m+1}(L_{v,n})} \\&=&
    C2^{-2mn} \|u\|^2_{\mathcal M_{\ba+\bone}^{m+1}(L_{v,n})}=Ch^{2m} \|u\|^2_{\mathcal M_{\ba+\bone}^{m+1}(L_{v,n})}.
\end{eqnarray*}
Then, the desired estimate follows by summing up the estimates from different layers $L_{v, i}$, $0\leq i\leq n$.
\end{proof}

\subsection{Estimates on  initial $e$-tetrahedra in $\maT_0$}

Throughout this subsection, let  $T_{(0)}:=\triangle^4x_0 x_{1} x_{2} x_{3}\in\maT_0$ be an $e$-tetrahedron with $x_0x_1$ on the edge $e\in\maE$ and let $\kappa_e$ be the associated grading parameter. 
Then, we define the mesh layer associated with $\maT_n$ on $T_{(0)}$ as follows.
\begin{definition} \label{def.elayer}(Mesh Layers in $e$-tetrahedra) Based on  Algorithm  \ref{alg.n1}, in each  refinement, an $e$-tetrahedron is cut by a parallelogram parallel to $x_0x_1$. For example, in the $e$-tetrahedron of Figure \ref{fig.n3}, the quadrilateral with vertices $x_{02},x_{12},x_{13},x_{03}$ is the aforementioned parallelogram. We denote by $P_{e,i}$ the parallelogram produced in the $i$th refinement, $1\leq i\leq n$. Therefore, the distance from $P_{e, i+1}$ to $e$ is $\kappa_e\times$ the distance from $P_{e, i}$ to $e$. For the mesh $\maT_n$, let the $i$th layer $L_{e,i}$ on $T_{(0)}$, $0<i<n$, be the region bounded by $P_{e,i}$, $P_{e,{i+1}}$, and the faces of $T_{(0)}$.   Define $L_{e,0}$  to be the sub-region of $T_{(0)}$  away from $e$ that is separated by $P_{e, 1}$. We define $L_{e, n}$ to be the sub-region of $T_{(0)}$ between $P_{e,n}$ and $e$.   See for example Figure \ref{fig.42}.   As in Definition \ref{def.vlayer},  the initial triangulation of the layer $L_{e,i}$, $0\leq i<n$, is the first decomposition of this region into tetrahedra. Thus, the initial triangulation of $L_{e,i}$ consists of those tetrahedra in $\maT_{i+1}$ that are contained in $L_{e,i}$.
\end{definition}
\begin{remark}\label{rk.57}
 In the mesh layers, the distance  $\rho_e$ to the edge $e$ satisfies
\be\label{eqn.rhoe}
\rho_e\sim \kappa_e^i \qquad {\rm{on}}\ L_{e,i}, \quad 0\leq i<n.
\ee  In addition, the mesh layers of an $e$-tetrahedron $T_{(0)}$ also satisfy the following properties (see Figure \ref{fig.42}):
\begin{itemize}
\item The layer $L_{e,i}$, $2\leq i\leq n$, is the union of two components: sub-regions from $2^{i-1}$ $e$-tetrahedra in $\maT_{i-1}$ and sub-regions from $2^{i}-2$ $v_e$-tetrahedra in $\maT_{i-1}$.
\item Among the aforementioned $2^i-2$ $v_e$-tetrahedra in $\maT_{i-1}$, $2^{k}$ of them are sub-regions of  $v_e$-tetrahedra  in $\maT_k$, $1\leq k\leq i-1$.
\end{itemize}
\end{remark}

Now, we start to develop some estimates for the shape regularity of the mesh on $L_{e, i}$, although it is in general  anisotropic and violates the maximum angle condition. These results will be used  for the interpolation error analysis.


\begin{definition} \label{def.rdis} (Relative Distances for $e$-tetrahedra) Recall the initial $e$-tetrahedron $T_{(0)}=\triangle^4 x_0x_1x_2x_3\in\maT_0$. For an $e$-tetrahedron $T=\triangle^4\gamma_0\gamma_1\gamma_2\gamma_3$  generated by some subsequent refinements of $ T_{(0)}$ based on Algorithm \ref{alg.n1}, consider its two vertices on the edge $x_0x_1$. We call the vertex that is closer to $x_0$ the \textit{first vertex} of $T$, and call the vertex closer to $x_1$ the \textit{second vertex} of $T$. 

Without loss of generality, we suppose  $\gamma_0\gamma_1\subset e\in\maE$ and $\gamma_0$ (resp. $\gamma_1$) is the \textit{first} (resp. \textit{second}) vertex  of $T$. Let $\gamma$ be either $\gamma_2$ or $\gamma_3$. Denote by $\gamma'$  the orthogonal projection of $\gamma$ on the $z$-axis (the axis containing the edge $e$). See for instance Figure \ref{fig.441}. Then, we define $c_{\gamma,1}$ to be the \textit{first relative $z$-distance} of $\gamma$, such that
\begin{eqnarray}\label{eqn.cg1}
|c_{\gamma,1}|=|{\gamma_{0}\gamma'}|/|{\gamma_{0}\gamma_1}|, \quad {\rm{and}}\quad
\left\{\begin{array}{ll}
 c_{\gamma,1}=|c_{\gamma,1}|\ \qquad{\rm{if}}\ \overrightarrow{\gamma_{0}\gamma'}=t (\overrightarrow{\gamma_{0}\gamma_1})\ {\rm{for\ some\ }} t>0 \\
c_{\gamma,1}=-|c_{\gamma,1}| \ \qquad {\rm{otherwise}}.
\end{array}\right.
\end{eqnarray}
The \textit{second relative $z$-distance} of $\gamma$, denoted by $c_{\gamma, 2}$, is defined by
\begin{eqnarray}\label{eqn.cg2}
|c_{\gamma, 2}|=|{\gamma_{1}\gamma'}|/|{\gamma_{0}\gamma_1}|, \quad {\rm{and}}\quad
\left\{\begin{array}{ll}
 c_{\gamma, 2}=|c_{\gamma, 2}|\ \qquad{\rm{if}}\ \overrightarrow{\gamma_{1}\gamma'}=t (\overrightarrow{\gamma_{1}\gamma_0})\ {\rm{for\ some\ }} t>0 \\
c_{\gamma, 2}=-|c_{\gamma, 2}| \ \qquad {\rm{otherwise}}.
\end{array}\right.
\end{eqnarray}
It is clear that $c_{\gamma, 2}=1-c_{\gamma, 1}$. In addition, we define the \textit{absolute relative distance} for $T$, denoted by $c_T$, such that 
\be\label{eqn.adis}
c_T=\max(|c_{\gamma_2,1}|, \ |c_{\gamma_2,2}|, \ |c_{\gamma_3,1}|, \ |c_{\gamma_3,2}|).
\ee
\end{definition}
\begin{remark} For each $e$-tetrahedron, there are four relative distances corresponding to the two vertices away from the $z$-axis. The sign of the relative distance is determined by the location of orthogonal projection of the off-the-edge vertex.  The relative distances imply, for the $e$-tetrahedron, how far the off-the-edge vertices shift away in the $z$-direction from the vertices on the $z$-axis.
\end{remark}

\begin{remark} \label{rk.cp1} Note that after one refinement, $T$ is decomposed into eight sub-tetrahedra: two $e$-tetrahedra (denoted by $T_A$ and $T_B$), two $v_e$-tetrahedra, and four $o$-tetrahedra.   In this case, we call $T$ the \textit{parent tetrahedron} of the sub-tetrahedra, and call each sub-tetrahedron  the \textit{child tetrahedron} of $T$. Note that Definition \ref{def.rdis}  is also valid for $ev$-tetrahedra. We shall use it later for $ev$-tetrahedra as well.
\end{remark}

In what follows, we establish the connections between $T$ and its child $e$-tetrahedra $T_A$ and $T_B$ in terms of the corresponding relative $z$-distances.

\begin{figure}
\includegraphics[scale=0.23]{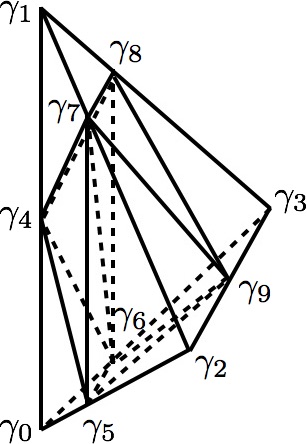} \hspace{4cm}
\includegraphics[scale=0.23]{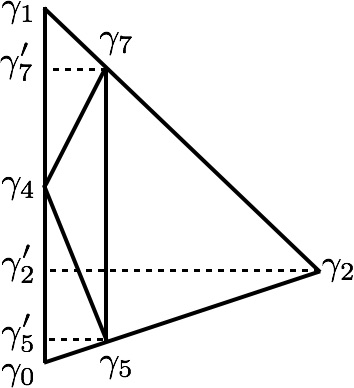} 
\caption{The mesh on an $e$-tetrahedron after one refinement (left); the induced triangles on an  face containing the singular edge (right).}\label{fig.441}
\end{figure}

\begin{lemma}\label{lem.451} Let $T\subset T_{(0)}$ be an $e$-tetrahedron in $\maT_i$, $1\leq i<n$. Let $T_A$, $T_B\subset T$ be the two child $e$-tetrahedra in $\maT_{i+1}$. Denote by $c_T$, $c_A$, and $c_B$ the absolute distances for $T$, $T_A$, and $T_B$ as in (\ref{eqn.adis}). Then, $\max(c_A, c_B)\leq \max(c_T,1)$.
\end{lemma}
\begin{proof} 
Denote $T$ by  $T=\triangle^4\gamma_0\gamma_1\gamma_2\gamma_3$ with the \textit{first} vertex $\gamma_0$ and the \textit{second} vertex $\gamma_1$ on the singular edge $\gamma_0\gamma_1$. As illustrated in Figure \ref{fig.441}, we  let $T_A:=\triangle^4\gamma_0\gamma_4\gamma_5\gamma_6$ and $T_B:=\triangle^4\gamma_1\gamma_4\gamma_7\gamma_8$. 
Recall the relative distance from Definition \ref{def.rdis}. In particular, let $c_{\gamma_2, 1}, c_{\gamma_2, 2}$ be the relative distances of $\gamma_2$ in $T$, and let $c^A_{\gamma_5, 1}, c^A_{\gamma_5, 2}$ (resp. $c^B_{\gamma_7, 1}, c^B_{\gamma_7, 2}$) be  the relative distances of $\gamma_5$ (resp. $\gamma_7$) in $T_A$ (resp. $T_B$). We first show $|c_{\gamma_5, 1}^A|, |c_{\gamma_5, 2}^A|\leq \max(|c_{\gamma_2, 1}|, |c_{\gamma_2, 2}|,1)$.

Consider the triangles on  the face $\triangle^3\gamma_0\gamma_1\gamma_2$ of $T$, induced by the sub-tetrahedra after one refinement of $T$ (the second picture in Figure \ref{fig.441}). In addition, we have drawn three  dashed line segments $\gamma_2\gamma_2'$, $\gamma_5\gamma_5'$, and $\gamma_7\gamma_7'$ that are perpendicular to ${\gamma_0\gamma_1}$. Then, by (\ref{eqn.cg1}), we have 
\begin{eqnarray*}
|c_{\gamma_2,1}|=|{\gamma_{0}\gamma_2'}|/|{\gamma_{0}\gamma_1}|, \quad {\rm{and}}\quad |c^A_{\gamma_5,1}|=|{\gamma_{0}\gamma_5'}|/|{\gamma_{0}\gamma_4}|.
\end{eqnarray*}
Note that $\triangle^3\gamma_0\gamma_5\gamma_5'$ is similar to $\triangle^3\gamma_0\gamma_2\gamma_2'$.  Therefore, $|{\gamma_0\gamma_5'}|=\kappa_e|{\gamma_{0}\gamma_2'}|$, and $c_{\gamma_2,1}$ and $c^A_{\gamma_5,1}$ have the same sign. Recall $0<\kappa_e\leq 1/2$. Then, we consider all the possible cases.

In the case $c_{\gamma_2,1}<0$, we have
$$
0> c^A_{\gamma_5,1}=-|{\gamma_{0}\gamma_5'}|/|{\gamma_{0}\gamma_4}|=-2\kappa_e|{\gamma_{0}\gamma_2'}|/|{\gamma_{0}\gamma_1}|=2\kappa_e c_{\gamma_2,1}\geq c_{\gamma_2,1}.
$$
Therefore, $|c^A_{\gamma_5,1}|\leq |c_{\gamma_2,1}|$.
Meanwhile, we have $$1\leq c^A_{\gamma_5,2}=1-c^A_{\gamma_5,1}=1-2\kappa_e c_{\gamma_2,1}<1-c_{\gamma_2,1}=c_{\gamma_2,2}.$$
Therefore, $|c^A_{\gamma_5,2}|\leq |c_{\gamma_2,2}|$.


In the case $0\leq c_{\gamma_2,1}<(2\kappa_e)^{-1}$, we have $c^A_{\gamma_5,1}\geq 0$ and
$$
c^A_{\gamma_5,1}=|{\gamma_{0}\gamma_5'}|/|{\gamma_{0}\gamma_4}|=2\kappa_e|{\gamma_{0}\gamma_2'}|/|{\gamma_{0}\gamma_1}|=2\kappa_e c_{\gamma_2,1}<1.
$$
Meanwhile, we have $$0\leq c^A_{\gamma_5,2}=1-c^A_{\gamma_5,1}=1-2\kappa_e c_{\gamma_2,1}< 1.$$

In the case $ c_{\gamma_2,1}\geq (2\kappa_e)^{-1}$, we have
$$
1\leq c^A_{\gamma_5,1}=|{\gamma_{0}\gamma_5'}|/|{\gamma_{0}\gamma_4}|=2\kappa_e|{\gamma_{0}\gamma_2'}|/|{\gamma_{0}\gamma_1}|=2\kappa_e c_{\gamma_2,1}\leq |c_{\gamma_2,1}|.
$$
Meanwhile, we have $$0\geq c^A_{\gamma_5,2} =1-c^A_{\gamma_5,1} =1-2\kappa_e c_{\gamma_2,1}\geq 1-c_{\gamma_2,1}=c_{\gamma_2,2}.$$
Therefore, $|c^A_{\gamma_5,2}|\leq |c_{\gamma_2,2}|$. Thus, we have shown 
$$|c^A_{\gamma_5,1}|, |c^A_{\gamma_5,2}|\leq \max(|c_{\gamma_2,1}|, |c_{\gamma_2,2}|,1).$$

With a similar calculation, we can derive the upper bounds for other relative distances in $T_A$ and $T_B$, namely, 
\ben
&|c^A_{\gamma_6,1}|, |c^A_{\gamma_6,2}|\leq \max(|c_{\gamma_3,1}|, |c_{\gamma_3,2}|,1), \\
&|c^B_{\gamma_7,1}|, |c^B_{\gamma_7,2}|\leq \max(|c_{\gamma_2,1}|, |c_{\gamma_2,2}|,1), \quad |c^B_{\gamma_8,1}|, |c^B_{\gamma_8,2}|\leq \max(|c_{\gamma_3,1}|, |c_{\gamma_3,2}|,1).\een
Hence, the proof is completed by (\ref{eqn.adis}).
\end{proof}

Recall that for a $v$- or $v_e$-tetrahedron in $\maT_0$, the isotropic transformation (\ref{eqn.bv1}) maps a mesh layer to a reference domain (either the tetrahedron itself or the layer $L_{v,0}$). Here, we define the reference domain for an $e$-tetrahedron.

\begin{definition}\label{def.rt} (The Reference $e$-tetrahedron) For the initial $e$-tetrahedron $T_{(0)}:=\triangle^4x_0x_1x_2x_3\in\maT_0$,   we use a local Cartesian coordinate system, such that the $z$-axis contains the edge $x_0x_1$ with the direction of $\overrightarrow{x_0x_1}$ as the positive direction, and $x_2$ is in the $xz$-plane. We will specify the origin later. Let $l_0:=|{x_0x_1}|$ be the length of the singular edge.  Then, we define the reference tetrahedron $\hat T=\triangle^4\hat x_0\hat x_1\hat x_2\hat x_3$, such that 
$$\hat x_0=(0, 0, -l_0/2), \quad\hat x_1=(0, 0, l_0/2), \quad\hat x_k=(\hat \lambda_k,\hat \xi_k, -l_0/2),\ k=2, 3,$$ 
where $\hat\lambda_k, \hat\xi_k$ are the $x$- and $y$-components of the vertices $x_2$ and $x_3$, respectively.  Therefore, $\hat\xi_2=0$ and $\hat \lambda_2, \hat \lambda_3, \hat\xi_3$ are constants that depend on the shape regularity of $T_{(0)}$. Thus, $\hat T$ is a  tetrahedron with one face in the plane $z=-l_0/2$, one face in the $xz$-plane, such that $|{\hat x_0\hat x_1}|= |{x_0x_1}|$, $|{\hat x_0\hat x_2}|=$ the length of the orthogonal projection of ${x_0x_2}$ in the plane $z=-l_0/2$, and  $|{\hat x_0\hat x_3}|=$ the length of the orthogonal projection of ${x_0x_3}$ in the plane $z=-l_0/2$.  In addition, we denote by $\hat\maT_1$ and $\hat\maT_2$ the triangulations of $\hat T$ after one and two edge refinements with parameter $\kappa_e$, respectively. See Figure \ref{fig.42} for example.
\end{definition}
\begin{figure}
\includegraphics[scale=0.16]{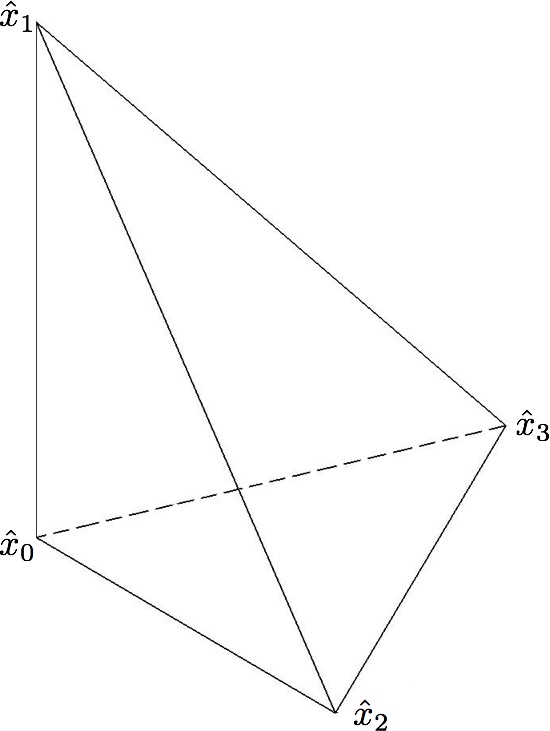} \hspace{.8cm}
\includegraphics[scale=0.16]{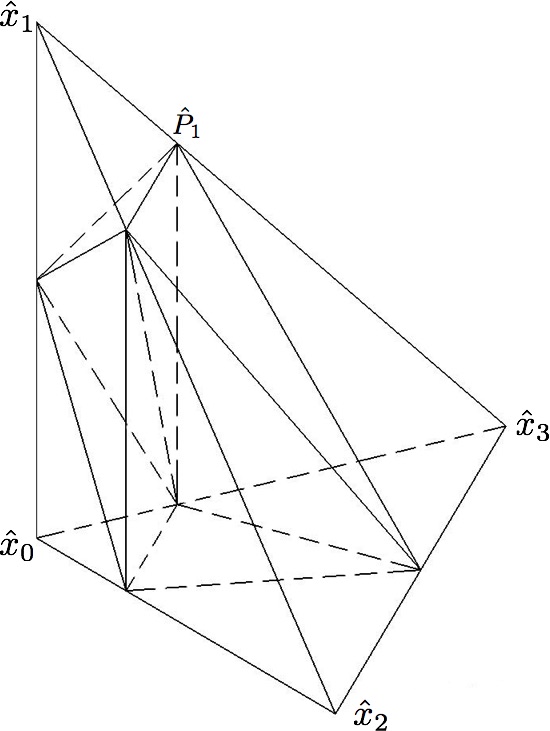}\hspace{.8cm}
\includegraphics[scale=0.16]{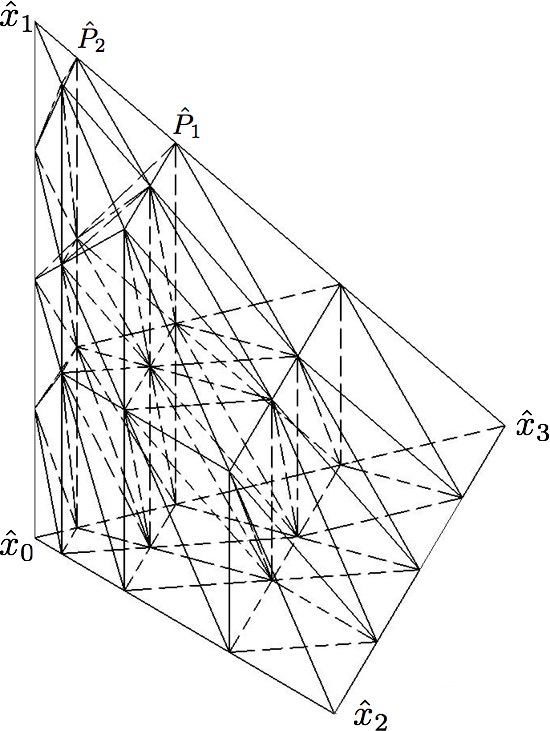}\hspace{.8cm}
\caption{\small{A reference tetrahedron $\hat T$ (left); the triangulation $\mathcal{\hat T}_1$ after one refinement (center); the triangulation $\mathcal{\hat T}_2$ after two edge refinements (right).}}\label{fig.42}
\end{figure}

In the following lemmas, we construct explicit linear mappings between an $e$-tetrahedron $\subset T_{(0)}$ and the reference tetrahedron $\hat T$.

\begin{lemma}\label{lem.46}
For an $e$-tetrahedron $\maT_i\ni T:=\triangle^4\gamma_0\gamma_1\gamma_2\gamma_3\subset T_{(0)}$, $1\leq i\leq n$, suppose $\gamma_0\gamma_1$ is the singular edge, and $\overrightarrow {\gamma_0\gamma_1}$ and  $\overrightarrow {x_0x_1}$ share the same direction. We use the local coordinate system in Definition \ref{def.rt}, and set $(\gamma_0+\gamma_1)/2$ to be the origin.  Then, there exist a matrix
\begin{eqnarray}\label{eqn.mapt}
\mathbf B_{e, i}= \begin{pmatrix}
   \kappa^{-i}_{e}   &   0 & 0 \\
  0    &   \kappa^{-i}_{e} & 0\\
  b_1 \kappa^{-i}_{e}\ & b_2 \kappa^{-i}_{e} & 2^{i}
\end{pmatrix}
\end{eqnarray}
with $|b_1|,|b_2|\leq C_0$, where $C_0> 0$ depends on the initial tetrahedron $T_{(0)}$  but not on $i$, 
such that $\mathbf B_{e, i}: T\rightarrow \hat T$ is a bijection. 
\end{lemma}
\begin{proof}Based on the refinement  in Algorithm \ref{alg.n1}, and on Defintion \ref{def.rt},  we have $\gamma_2=(\kappa^{i}_e\hat\lambda_2, 0, \zeta_2)$,  $\gamma_3=(\kappa^{i}_e\hat\lambda_3,\kappa^{i}_e\hat\xi_3, \zeta_3$), and   $|{\gamma_0\gamma_1}|=2^{-i}l_0$=$2^{-i}|x_0x_1|$, where $\zeta_2$ and $\zeta_3$ are the $z$-coordinates of the vertices $\gamma_2$ and $\gamma_3$, respectively.  Thus, the anisotropic transformation 
\begin{eqnarray*}
\mathbf A_1:= \begin{pmatrix}
   \kappa^{-i}_{e}   &   0 & 0 \\
  0    &   \kappa^{-i}_{e} & 0\\
  0 & 0 & 2^{i}
\end{pmatrix}
\end{eqnarray*}
maps $T$ to a tetrahedron, with vertices $\mathbf A_1\gamma_0=(0, 0, -l_0/2)$,  $\mathbf A_1\gamma_1=(0, 0, l_0/2)$, $\mathbf A_1\gamma_2=(\hat\lambda_2, 0, 2^{i}\zeta_2)$, and $\mathbf A_1\gamma_3=(\hat\lambda_3, \hat\xi_3, 2^{i}\zeta_3)$. Now, define
\be\label{eqn.mapt1}b_1=-(2^{i}\zeta_2+2^{-1}l_0)/\hat\lambda_2, \qquad b_2=[2^{i}(\zeta_2\hat\lambda_3-\hat\lambda_2\zeta_3)+2^{-1}l_0(\hat\lambda_3-\hat\lambda_2)]/\hat\lambda_2\hat\xi_3,
\ee
and let 
\begin{eqnarray*}\label{eqn.mapt3}
\mathbf A_2:= \begin{pmatrix}
   1   &   0 & 0 \\
  0    &   1 & 0\\
  b_1 & b_2 & 1
\end{pmatrix}.
\end{eqnarray*}
Then,  a straightforward calculation shows that 
\begin{eqnarray*}\label{eqn.mapt4}
\mathbf B_{e, i}:=\mathbf A_2\mathbf A_1= \begin{pmatrix}
   \kappa^{-i}_{e}   &   0 & 0 \\
  0    &   \kappa^{-i}_{e} & 0\\
  b_1 \kappa^{-i}_{e}\ & b_2 \kappa^{-i}_{e} & 2^{i}
\end{pmatrix}
\end{eqnarray*}
maps $T$ to $\hat T$. Meanwhile, by Lemma \ref{lem.451},  we have $|\zeta_2|, |\zeta_3|\leq C|{\gamma_0\gamma_1}|=C2^{-i}l_0$, where $C$ depends on the shape regularity of $T_{(0)}$. In addition, since $\hat\lambda_2, \hat\lambda_3, \hat\xi_3$ all depend on the shape regularity of $T_{(0)}$ and $\hat\lambda_2,\hat\xi_3\neq 0$,  by (\ref{eqn.mapt1}), we conclude  $|b_1|,|b_2|\leq C_0$, where $C_0\geq 0$ depends on $T_{(0)}$ but not on $i$. 
 \end{proof}

Recall the parent and child tetrahedra associated to each mesh refinement in Remark \ref{rk.cp1}. Note that  for a $v_e$-tetrahedron $T_{(i)}\subset T_{(0)}$  in  $\maT_i$, its parent tetrahedron, which is in $\maT_{i-1}$, can be either a $v_e$-tetrahedron or an $e$-tetrahedron. Nevertheless, there exists a $v_e$-tetrahedron $T_{(k)}\in\maT_k$, $1\leq k\leq i$, such that $T_{(i)}\subset T_{(k)}\subset T_{(0)}$ and $T_{(k)}$'s parent tetrahedron is an $e$-tetrahedron in $\maT_{k-1}$.

Next, we construct the mapping between a $v_e$-tetrahedron in $\maT_i$ and the reference domain. Recall the triangulations $\hat\maT_1$ and $\hat\maT_2$ of $\hat T$ in Definition \ref{def.rt}.

\begin{lemma} \label{lem410} Let  $T_{(i)}\subset T_{(0)}$ be  a $v_e$-tetrahedron in $\maT_i$, $1\leq i\leq n$. Let $T_{(k)}\in \maT_k$, $1\leq k\leq i$, be the $v_e$-tetrahedron, such that $T_{(i)}\subset T_{(k)}$  and  $T_{(k)}$'s parent tetrahedron $T_{(k-1)}=\triangle^4\gamma_0\gamma_1\gamma_2\gamma_3\in\maT_{k-1}$ is an $e$-tetrahedron. On $T_{(k-1)}$, we use the same local coordinate system as in Lemma \ref{lem.46}   with origin at $(\gamma_0+\gamma_1)/2$.  Then, there is a transformation
 \begin{eqnarray}\label{eqn.bs5}
 \mathbf B_{i, k}= \begin{pmatrix}
   \kappa^{-i+1}_{e}   &   0 & 0 \\
  0    &   \kappa^{-i+1}_{e} & 0\\
  b_1 \kappa^{-i+1}_{e}\ & b_2 \kappa^{-i+1}_{e} & 2^{k-1}\kappa_e^{k-i}
\end{pmatrix}
\end{eqnarray}
that maps $T_{(i)}$  to a $v_e$-tetrahedron in $\mathcal{\hat T}_1$, where  $|b_1|,|b_2|\leq C_0$, for $C_0> 0$ depending on $T_{(0)}$ but not on $i$ or $k$.
\end{lemma}
\begin{proof}
Based on Algorithm \ref{alg.n1}, the origin $(\gamma_0+\gamma_1)/2$ is  the vertex of $T_{(i)}$ on the singular edge. Then, the linear mapping
 \begin{eqnarray*}
 \mathbf A_1= \begin{pmatrix}
   \kappa^{k-i}_{e}   &   0 & 0 \\
  0    &   \kappa^{k-i}_{e} & 0\\
 0 & 0 & \kappa_e^{k-i}
\end{pmatrix}
\end{eqnarray*}
translates $T_{(i)}$ to $T_{(k)}$. Since $T_{(k-1)}$ is an $e$-tetrahedron, by Lemma \ref{lem.46}, the transformation 
 \begin{eqnarray*}
 \mathbf A_2= \begin{pmatrix}
   \kappa^{-k+1}_{e}   &   0 & 0 \\
  0    &   \kappa^{-k+1}_{e} & 0\\
  b_1 \kappa^{-k+1}_{e}\ & b_2 \kappa^{-k+1}_{e} & 2^{k-1}
\end{pmatrix}
\end{eqnarray*}
maps $T_{(k-1)}$ to $\hat T$, and also maps the restriction of $\maT_k$ on $T_{(k-1)}$ to $\hat\maT_1$, where $|b_1|, |b_2|<C$ for $C$  depending on $T_{(0)}$. Therefore,
\begin{eqnarray*}
\mathbf B_{i, k}:=\mathbf A_2\mathbf A_1= \begin{pmatrix}
   \kappa^{-i+1}_{e}   &   0 & 0 \\
  0    &   \kappa^{-i+1}_{e} & 0\\
  b_1 \kappa^{-i+1}_{e}\ & b_2 \kappa^{-i+1}_{e} & 2^{k-1}\kappa_e^{k-i}
\end{pmatrix}
\end{eqnarray*}  
maps $T_{(i)}$ to one of the $v_e$-tetrahedra in $\hat\maT_1$. This completes the proof.
\end{proof}

Now, we are ready to construct the mapping from a tetrahedron $T_{(i+1)}\in\maT_{i+1}$ in the mesh layer $L_{e,i}$ (Definition \ref{def.elayer}) to the reference domain. Also recall that  $T_{(i+1)}$ is a tetrahedron in the initial triangulation of $L_{e, i}$.

\begin{lemma} \label{cor.46}Let $T_{(i+1)}\in\maT_{i+1}$ be a tetrahedron, such that  $T_{(i+1)}\subset L_{e,i}\subset T_{(0)}$, $0\leq i< n$. \\
Case I: $T_{(i+1)}$ is a child tetrahedron of an $e$-tetrahedron $T_{(i)}\in\maT_{i}$. Using the  $T_{(i)}$-based local coordinate system as in Lemma \ref{lem.46},  the transformation 
 \begin{eqnarray}\label{eqn.bs1}
 \mathbf B_{e, i}= \begin{pmatrix}
   \kappa^{-i}_{e}   &   0 & 0 \\
  0    &   \kappa^{-i}_{e} & 0\\
  b_1 \kappa^{-i}_{e}\ & b_2 \kappa^{-i}_{e} & 2^{i}
\end{pmatrix}
\end{eqnarray}
 maps $T_{(i+1)}$  to some $o$-tetrahedron in $\mathcal{\hat T}_1$. \\
Case II:  $T_{(i+1)}$ is a child tetrahedron of a $v_e$-tetrahedron $T_{(i)}\in\maT_{i}$. Let $T_{(k)}\in \maT_k$, $1\leq k\leq i$, be the $v_e$-tetrahedron, such that $T_{(i)}\subset T_{(k)}$  and  $T_{(k)}$'s parent tetrahedron $T_{(k-1)}\in\maT_{k-1}$ is an $e$-tetrahedron.  Using the  $T_{(k-1)}$-based local coordinate system as in Lemma \ref{lem410}, the transformation
 \begin{eqnarray}\label{eqn.bs2}
 \mathbf B_{i, k}= \begin{pmatrix}
   \kappa^{-i+1}_{e}   &   0 & 0 \\
  0    &   \kappa^{-i+1}_{e} & 0\\
  b_1 \kappa^{-i+1}_{e}\ & b_2 \kappa^{-i+1}_{e} & 2^{k-1}\kappa_e^{k-i}
\end{pmatrix}
 \end{eqnarray}
maps $T_{(i+1)}$  to an $o$-tetrahedron in $\mathcal{\hat T}_2$. 
 In both cases,   $|b_1|,|b_2|\leq C_0$, for $C_0> 0$ depending on $T_{(0)}$ but not on $i$ or $k$.
\end{lemma}
  \begin{proof}
If $T_{(i+1)}$ is a child tetrahedron of an $e$-tetrahedron $T_{(i)}\in\maT_{i}$, the matrix in (\ref{eqn.mapt}) maps $T_{(i)}$ to $\hat T$, and maps $P_{e,i+1}\cap T_{(i)}$ to $\hat P_1$, where $P_{e, i+1}$ is the parallelogram cutting $T_{(0)}$ in the $i+1$st refinement (Definition \ref{def.elayer}) and $\hat P_1$ is the parallelogram cutting $\hat T$ in the first edge refinement (see Figure \ref{fig.42}). Consequently,  $T_{(i+1)}$ is translated to one of the four $o$-tetrahedra  in $\mathcal{\hat T}_1$ by the same mapping. 

For Case II, the transformation (\ref{eqn.bs5}) maps $T_{(i)}$ to a $v_e$-tetrahedron in $\hat\maT_1$. In addition, it maps $P_{e, i}\cap \bar T_{(i)}$ to $\hat P_1$, and $P_{e, i+1}\cap T_{(i)}$ to $\hat P_2$ (see Figure \ref{fig.42}). Therefore, the same transformation maps $T_{(i+1)}$ to an $o$-tetrahedron in $\hat\maT_2$ between $\hat P_1$ and $\hat P_2$. This completes the proof.
\end{proof}

Each tetrahedron  in $\maT_{i+1}$ that belongs to layer $L_{e,i}$ falls into either Case I or Case II of Lemma \ref{cor.46}. Thus, there is a linear transformation $\mathbf B$ (either $\mathbf B_{e, i}$ or $\mathbf B_{i,k}$) that maps $T_{(i+1)}$ to an $o$-tetrahedron in either $\hat\maT_1$ or in $\hat\maT_2$. We denote this $o$-tetrahedron by $\hat T_{(i+1)}$. It is clear that $\hat T_{(i+1)}$ belongs to a finite number of similarity classes determined by the $o$-tetrahedra in $\hat\maT_1$ and  $\hat\maT_2$.  Then, for $(x,y ,z)\in T_{(i+1)}$, we have
\begin{equation}\label{eqn.dilation}
\mathbf B(x,y ,z)=(\hat x,\hat y, \hat z)\in \hat T_{(i+1)}.
\end{equation}
For a function $v$ on $T_{(i+1)}$, we define  $\hat v(\hat x,\hat y, \hat z):=v(x,y ,z)$.

In the $i+1$st refinement, $0\leq i<n$, when the layer $L_{e,i}$ is formed, it only contains tetrahedra  in $\maT_{i+1}$. To obtain the mesh $\maT_n$, these  tetrahedra  in $L_{e,i}$ are further refined uniformly $n-i-1$ times. In the following, we obtain a uniform interpolation error estimate for the mesh $\maT_n$ in the layer $L_{e,i}$.
\begin{theorem} \label{thm.11}
Let $T_{(0)}\in\maT_{0}$ be an $e$-tetrahedron. For $u\in\mathcal M^{m+1}_{\ba+\bone}(\Omega)$, where $\ba$ is given in (\ref{eqn.aell}), let $u_I$ be its nodal interpolation on $\maT_n$. Then, for $0\leq i<n$, we have 
$$
|u-u_I|_{H^1(L_{e,i})}\leq Ch^{m}\|u\|_{\mathcal M^{m+1}_{\ba+\bone}(L_{e,i})},
$$
where $L_{e,i}$ is the mesh layer in Definition \ref{def.elayer}, $h=2^{-n}$, and $C$ depends on $T_{(0)}$ and $m$, but not on $i$.
\end{theorem}

\begin{proof} Based on Algorithm \ref{alg.n1}, 
the layer $L_{e,i}$ is formed in the $i+1$st refinement and is the union of tetrahedra in $\maT_{i+1}$ between $P_{e,i}$ and $P_{e,i+1}$. Therefore, it suffices to verify the following interpolation error estimate on each tetrahedron $\maT_{i+1}\ni T_{(i+1)}\subset L_{e,i}$,
\be\label{eqn.need1}
|u-u_I|_{H^1(T_{(i+1)})}\leq Ch^{m}\|u\|_{\mathcal M^{m+1}_{\ba+\bone}(T_{(i+1)})}.
\ee
We show this estimate based on the type of $T_{(i+1)}$'s parent tetrahedron.

Case I: $T_{(i+1)}$'s parent  is an $e$-tetrahedron in $\maT_i$.  Let  $(x,y,z)\in T_{(i+1)}$ and $(\hat x,\hat y, \hat z)\in\hat T_{(i+1)}$ as in (\ref{eqn.dilation}). Then, by the mapping in (\ref{eqn.bs1}) and  direct calculation, we have
\be\label{eqn.d111}
\left\{\begin{array}{ll}
dxdydz=2^{-i}\kappa_e^{2i}d\hat xd\hat yd\hat z; \\
\pa_{ x} v=(\kappa_e^{-i}\pa_{\hat x}+b_1\kappa_e^{-i}\pa_{\hat z})\hat v,\quad  \pa_{y} v=(\kappa_e^{-i}\pa_{\hat y}+b_2\kappa_e^{-i}\pa_{\hat z})\hat v, \quad \pa_{z} v=2^i\pa_{\hat z}\hat v; \\
\pa_{\hat x} \hat v=(\kappa_e^{i}\pa_{x}-b_12^{-i}\pa_{ z}) v,\quad \pa_{\hat y} \hat v=(\kappa_e^{i}\pa_{y}-b_22^{-i}\pa_{ z}) v, \quad \pa_{\hat z} \hat v=2^{-i}\pa_{ z} v.
\end{array}\right.
\ee
Therefore, by Lemma \ref{cor.46}, (\ref{eqn.d111}),  the standard interpolation estimate on $\hat T_{(i+1)}$, (\ref{eqn.rhoe}), and (\ref{eqn.ke1}), we have
 \begin{eqnarray*}
\|\partial_x(u-u_I)\|_{L^2(T_{(i+1)})}^2&\leq& C2^{-i}\big(\|\partial_{\hat x}(\hat u-\hat u_I)\|_{L^2(\hat T_{(i+1)})}^2+\|\partial_{\hat z}(\hat u-\hat u_I)\|^2_{L^2(\hat T_{(i+1)})}\big)\\
&\leq& C2^{-i}2^{2m(i-n)}|\hat u|^2_{H^{m+1}(\hat T_{(i+1)})}  \\
&\leq& C2^{2m(i-n)}\sum_{|\alpha_\perp|+\alpha_3= m+1}2^{-2i\alpha_3}\kappa_e^{2i(|\alpha_\perp|-1)}\|\partial^{\alpha_\perp}\partial_z^{\alpha_3}u\|^2_{L^2(T_{(i+1)})}\\
&\leq& C2^{2m(i-n)}\sum_{|\alpha_\perp|+\alpha_3= m+1}2^{-2i\alpha_3}\|\rho_{e}^{|\alpha_\perp |-1}\partial^{\alpha_\perp}\partial_z^{\alpha_3}u\|^2_{L^2(T_{(i+1)})}\\
&\leq &C2^{2m(i-n)}\kappa_{e}^{2ia_{e}}\|u\|_{\mathcal M^{m+1}_{\ba+\bone}(T_{(i+1)})}^{2} \leq Ch^{2m}\|u\|_{\mathcal M^{m+1}_{\ba+\bone}(T_{(i+1)})}^2.\nonumber
\end{eqnarray*}
A similar calculation for the derivative with respect to $y$ gives
\ben
\|\partial_y(u-u_I)\|_{L^2(T_{(i+1)})}\leq Ch^{m}\|u\|_{\mathcal M^{m+1}_{\ba+\bone}(T_{(i+1)})}.
\een
In the $z$-direction, by Lemma \ref{cor.46}, (\ref{eqn.d111}), the standard interpolation estimate, (\ref{eqn.rhoe}), and (\ref{eqn.ke1}),  we have
\ben
\|\partial_z(u-u_I)\|_{L^2(T_{(i+1)})}^2&\leq& C2^{i}\kappa_e^{2i}\|\partial_{\hat z}(\hat u-\hat u_I)\|^2_{L^2(\hat T_{(i+1)})}\\
&\leq& C2^{i}\kappa_e^{2i}2^{2m(i-n)}|\hat u|^2_{H^{m+1}(\hat T_{(i+1)})}  \\
&\leq& C2^{2m(i-n)}\sum_{|\alpha_\perp|+\alpha_3= m+1}2^{2i}2^{-2i\alpha_3}\kappa_e^{2i|\alpha_\perp|}\|\partial^{\alpha_\perp}\partial_z^{\alpha_3}u\|^2_{L^2(T_{(i+1)})}\\
&\leq& C2^{2m(i-n)}\sum_{|\alpha_\perp|+\alpha_3= m+1}2^{-2i\alpha_3}\|\rho_{e}^{|\alpha_\perp |-1}\partial^{\alpha_\perp}\partial_z^{\alpha_3}u\|^2_{L^2(T_{(i+1)})}\\
&\leq &C2^{2m(i-n)}\kappa_{e}^{2ia_{e}}\|u\|_{\mathcal M^{m+1}_{\ba+\bone}(T_{(i+1)})}^{2} \leq Ch^{2m}\|u\|_{\mathcal M^{m+1}_{\ba+\bone}(T_{(i+1)})}^2.\nonumber
\een
Hence, we have completed the proof for (\ref{eqn.need1}).

Case II: $T_{(i+1)}$'s parent is  a $v_e$-tetrahedron $T_{(i)}\in\maT_{i}$. Let $T_{(k)}\in \maT_k$, $1\leq k\leq i$, be the $v_e$-tetrahedron, such that $T_{(i)}\subset T_{(k)}$  and  $T_{(k)}$'s parent tetrahedron $T_{(k-1)}\in\maT_{k-1}$ is an $e$-tetrahedron. Then, using the mapping (\ref{eqn.bs2}), by (\ref{eqn.dilation}), for $(x,y,z)\in T_{(i+1)}$ and $(\hat x,\hat y, \hat z)\in\hat T_{(i+1)}$,  we have
\be\label{eqn.etr2}
\hspace{0.5cm}\left\{\begin{array}{ll}
dxdydz=2^{1-k}\kappa_e^{3i-k-2}d\hat xd\hat yd\hat z;\\
\pa_{ x} v=(\kappa_e^{1-i}\pa_{\hat x}+b_1\kappa_e^{1-i}\pa_{\hat z})\hat v,\quad  \pa_{y} v=(\kappa_e^{1-i}\pa_{\hat y}+b_2\kappa_e^{1-i}\pa_{\hat z})\hat v, \quad \pa_{z} v=2^{k-1}\kappa_e^{k-i}\pa_{\hat z}\hat v; \\
\pa_{\hat x} \hat v=(\kappa_e^{i-1}\pa_{x}-b_12^{1-k}\kappa_e^{i-k}\pa_{ z}) v,\quad \pa_{\hat y} \hat v=(\kappa_e^{i-1}\pa_{y}-b_22^{1-k}\kappa_e^{i-k}\pa_{ z}) v, \quad \pa_{\hat z} \hat v=2^{1-k}\kappa_e^{i-k}\pa_{ z} v.
\end{array}\right.
\ee
Therefore, by Lemma \ref{cor.46}, (\ref{eqn.etr2}),   the standard interpolation estimate, (\ref{eqn.rhoe}), and (\ref{eqn.ke1}),  we have
 \begin{eqnarray*}
\|\partial_x(u-u_I)\|_{L^2(T_{(i+1)})}^2&\leq& C2^{1-k}\kappa_e^{i-k}\big(\|\partial_{\hat x}(\hat u-\hat u_I)\|_{L^2(\hat T_{(i+1)})}^2+\|\partial_{\hat z}(\hat u-\hat u_I)\|^2_{L^2(\hat T_{(i+1)})}\big)\\
&\leq& C2^{1-k}\kappa_e^{i-k}2^{2m(i-n)}|\hat u|^2_{H^{m+1}(\hat T_{(i+1)})}  \\
&\leq& C2^{2m(i-n)}\sum_{|\alpha_\perp|+\alpha_3= m+1}2^{2(1-k)\alpha_3}\kappa_e^{2(i-k)\alpha_3}\kappa_e^{(2i-2)(|\alpha_\perp|-1)}\|\partial^{\alpha_\perp}\partial_z^{\alpha_3}u\|^2_{L^2(T_{(i+1)})}\\
&\leq& C2^{2m(i-n)}\sum_{|\alpha_\perp|+\alpha_3= m+1}2^{2(1-k)\alpha_3}\|\rho_{e}^{|\alpha_\perp |-1}\partial^{\alpha_\perp}\partial_z^{\alpha_3}u\|^2_{L^2(T_{(i+1)})}\\
&\leq &C2^{2m(i-n)}\kappa_{e}^{2ia_{e}}\|u\|_{\mathcal M^{m+1}_{\ba+\bone}(T_{(i+1)})}^{2} \leq Ch^{2m}\|u\|_{\mathcal M^{m+1}_{\ba+\bone}(T_{(i+1)})}^2.\nonumber
\end{eqnarray*}
A similar calculation for the derivative with respect to $y$ gives
\ben
\|\partial_y(u-u_I)\|_{L^2(T_{(i+1)})}\leq Ch^{m}\|u\|_{\mathcal M^{m+1}_{\ba+\bone}(T_{(i+1)})}.
\een
In the $z$-direction,  by Lemma \ref{cor.46}, (\ref{eqn.etr2}),  the standard interpolation estimate, (\ref{eqn.rhoe}), and (\ref{eqn.ke1}), we have
\ben
\|\partial_z(u-u_I)\|_{L^2(T_{(i+1)})}^2&\leq& C(2^{1-k}\kappa_e^{i-k})\kappa_e^{2(i-1)}(2^{k-1}\kappa_e^{k-i})^2\|\partial_{\hat z}(\hat u-\hat u_I)\|^2_{L^2(\hat T_{(i+1)})}\\
&\leq& C(2^{1-k}\kappa_e^{i-k})\kappa_e^{2(i-1)}(2^{k-1}\kappa_e^{k-i})^22^{2m(i-n)}|\hat u|^2_{H^{m+1}(\hat T_{(i+1)})}  \\
&\leq& C2^{2m(i-n)}\sum_{|\alpha_\perp|+\alpha_3= m+1}(2^{1-k}\kappa_e^{i-k})^{2(\alpha_3-1)}\kappa_e^{2|\alpha_\perp|(i-1)}\|\partial^{\alpha_\perp}\partial_z^{\alpha_3}u\|^2_{L^2(T_{(i+1)})}\\
&\leq& C2^{2m(i-n)}\sum_{|\alpha_\perp|+\alpha_3= m+1}(2^{1-k}\kappa_e^{i-k})^{2(\alpha_3-1)}\kappa_e^{2i-2|\alpha_\perp|}\|\rho_{e}^{|\alpha_\perp |-1}\partial^{\alpha_\perp}\partial_z^{\alpha_3}u\|^2_{L^2(T_{(i+1)})}\\
&\leq &C2^{2m(i-n)}\kappa_{e}^{2ia_{e}}\|u\|_{\mathcal M^{m+1}_{\ba+\bone}(T_{(i+1)})}^{2} \leq Ch^{2m}\|u\|_{\mathcal M^{m+1}_{\ba+\bone}(T_{(i+1)})}^2.\nonumber
\een
This completes the proof for (\ref{eqn.need1}) of Case II. 

Hence, the theorem is proved by summing up the estimates for all the tetrahedra $T_{(i+1)}$ in $L_{e, i}$.
\end{proof}

Then, we extend the interpolation error estimate to the entire initial tetrahedron $T_{(0)}\in\maT_0$.
\begin{corollary}\label{cor.ete}
Let $T_{(0)}\in\maT_{0}$ be an $e$-tetrahedron. For $u\in\mathcal M^{m+1}_{\ba+\bone}(\Omega)$, where $\ba$ is given in (\ref{eqn.aell}), let $u_I$ be its nodal interpolation on $\maT_n$. Then, we have
$$
|u-u_I|_{H^1(T_{(0)})}\leq Ch^{m}\|u\|_{\mathcal M^{m+1}_{\ba+\bone}(T_{(0)})},
$$
where $h=2^{-n}$ and $C$ depends on $T_{(0)}$ and $m$.
\end{corollary}
\begin{proof}
By Theorem \ref{thm.11}, it suffices to show the estimate for any tetrahedron $T_{(n)}\in\maT_n$ in the last layer $L_{e,n}$. We derive the desired estimate in the following two cases.

Case I: $T_{(n)}$ is an $e$-tetrahedron. By Lemma \ref{lem.46}, the mapping $\mathbf B_{e,n}$ translates $T_{(n)}$ to the reference tetrahedron $\hat T$. Consequently, it maps any point $(x,y , z)\in T_{(n)}$ to $(\hat x, \hat y, \hat z)\in \hat T$. For a function $v$ on $T_{(n)}$, we define $\hat v$ on $\hat T$ by
$$
\hat v(\hat x, \hat y, \hat z):=v(x, y, z).
$$
Now, let $\chi$ be a smooth cutoff function on $\hat T$ such that
$\chi=0$ in a neighborhood of the edge $\hat e:=\hat x_0\hat x_1$ and $=1$ at every other Lagrange node of
$\hat T$. Let $\rho_{\hat{e}}$ be the distance to $\hat e$.   Let $\hat u_I$ be the interpolation of $\hat u$ on the  reference tetrahedron $\hat T$.  Since $\chi\hat u=0$ in the neighborhood of $\hat e$, $(\chi\hat u)_{I}=\hat u_I$ and 
\be\label{eqn.equ1111}
|\chi\hat u|_{H^{m+1}(\hat T)}^2\leq C\sum_{|\alpha_\perp|+\alpha_3\leq m+1}\|\rho_{\hat e}^{|\alpha_\perp|-1}\pa^{\alpha_\perp}\pa_{\hat z}^{\alpha_3}\hat u\|^2_{L^2(\hat T)}.
\ee Define $\hat w:=\hat u-\chi\hat u$.   Then, by  the usual interpolation error estimate, we have
\begin{eqnarray}
    |\hat u - \hat u_{I} |_{H^1
    (\hat T)}& =&  |\hat w+\chi \hat u - \hat u_{I} |_{H^1 (\hat T)}  \leq  |\hat w|_{H^1
    (\hat T)}+|\chi\hat u-\hat u_{I}|_{H^1
    (\hat T)}\nonumber\\
  &=& |\hat w|_{H^1 (\hat T)}+|\chi\hat u -
  (\chi\hat u)_{I}|_{H^1 (\hat T)}\leq C(\|\hat u\|_{H^1 (\hat T)}+|\chi\hat u|_{H^{m+1}(\hat T)}), \label{eqn.new1111111}
\end{eqnarray}
where $C$ depends on $m$ and, through $\chi$, the nodes on $\hat T$. Then, using   the scaling argument based on (\ref{eqn.d111}), (\ref{eqn.new1111111}), (\ref{eqn.equ1111}), the relation $\rho_{\hat e}(\hat x, \hat y,\hat z)=\kappa_e^{-n}\rho_e(x, y, z)$,   and (\ref{eqn.ke1}), we have
 \begin{eqnarray*}
\|\partial_x(u-u_I)\|_{L^2(T_{(n)})}^2&\leq& C2^{-n}\big(\|\partial_{\hat x}(\hat u-\hat u_I)\|_{L^2(\hat T)}^2+\|\partial_{\hat z}(\hat u-\hat u_I)\|^2_{L^2(\hat T)}\big)\\
&\leq& C2^{-n}\sum_{|\alpha_\perp|+\alpha_3\leq m+1}\|\rho_{\hat e}^{|\alpha_\perp|-1}\pa^{\alpha_\perp}\pa_{\hat z}^{\alpha_3}\hat u\|^2_{L^2(\hat T)}  \\
&\leq& C \sum_{|\alpha_\perp|+\alpha_3\leq m+1}2^{-2n\alpha_3}\|\rho_e^{|\alpha_\perp|-1}\partial^{\alpha_\perp}\partial_z^{\alpha_3}u\|^2_{L^2(T_{(n)})}\\
&\leq& C\sum_{|\alpha_\perp|+\alpha_3\leq m+1}2^{-2n\alpha_3}\kappa_{e}^{2na_e}\|\rho_{e}^{|\alpha_\perp |-1-a_e}\partial^{\alpha_\perp}\partial_z^{\alpha_3}u\|^2_{L^2(T_{(n)})}\\
&\leq &Ch^{2m}\|u\|^2_{\mathcal M^{m+1}_{\ba+\bone}(T_{(n)})}.\nonumber
\end{eqnarray*}
A similar calculation for the derivative with respect to $y$ gives
\ben
\|\partial_y(u-u_I)\|_{L^2(T_{(n)})}\leq Ch^{m}\|u\|_{\mathcal M^{m+1}_{\ba+\bone}(T_{(n)})}.
\een
In the $z$-direction, using (\ref{eqn.new1111111}), (\ref{eqn.equ1111}),  (\ref{eqn.d111}),   and (\ref{eqn.ke1}), we have
\ben
\|\partial_z(u-u_I)\|_{L^2(T_{(n)})}^2&=& 2^{n}\kappa_e^{2n}\|\partial_{\hat z}(\hat u-\hat u_I)\|^2_{L^2(\hat T)}\\
&\leq& C2^{n}\kappa_e^{2n}\sum_{|\alpha_\perp|+\alpha_3\leq m+1}\|\rho_{\hat e}^{|\alpha_\perp|-1}\pa^{\alpha_\perp}\pa_{\hat z}^{\alpha_3}\hat u\|^2_{L^2(\hat T)}   \\
&\leq& C\sum_{|\alpha_\perp|+\alpha_3\leq m+1}2^{-2n\alpha_3}\|\rho_e^{|\alpha_\perp|-1}\partial^{\alpha_\perp}\partial_z^{\alpha_3}u\|^2_{L^2(T_{(n)})}\\
&\leq &C\kappa_{e}^{2na_{e}}\|u\|_{\mathcal M^{m+1}_{\ba+\bone}(T_{(n)})}^{2} \leq Ch^{2m}\|u\|_{\mathcal M^{m+1}_{\ba+\bone}(T_{(n)})}^2.\nonumber
\een
Thus, we have proved the estimate for Case I.

Case II: $T_{(n)}$ is a $v_e$-tetrahedron. Let $T_{(k)}\in \maT_k$, $1\leq k\leq n$, be the $v_e$-tetrahedron, such that $T_{(n)}\subset T_{(k)}$  and  $T_{(k)}$'s parent tetrahedron $T_{(k-1)}\in\maT_{k-1}$ is an $e$-tetrahedron.  By Lemma \ref{lem410}, the mapping $\mathbf B_{n,k}$ translates $T_{(n)}$ to a $v_e$-tetrahedron in $\hat  T_{(n)}\in\hat \maT_1$.  Thus, $\mathbf B_{n, k}$ maps every point $(x, y, z)\in T_{(n)}$ to $(\hat x, \hat y, \hat z)\in\hat T_{(n)}$. As in Case I, for a function $v$ on $T_{(n)}$, we define $\hat v$ on $\hat T_{(n)}$ by
$$
\hat v(\hat x, \hat y, \hat z):=v(x, y, z).
$$
Now let $\chi$ be a smooth cutoff function on $\hat T_{(n)}$ such that
$\chi=0$ in a neighborhood of the  singular vertex on  $\hat e:={\hat x_0\hat x_1}$ of $\hat T$ and $=1$ at every other Lagrange node of
$\hat T_{(n)}$.  Recall the distance $\rho_{\hat e}$ to $\hat e$. Since $\chi\hat u=0$ in the neighborhood of the singular vertex,  we have $(\chi\hat u)_{I}=\hat u_I$ on $\hat  T_{(n)}$ and 
\begin{eqnarray}\label{eqn.mm}
|\chi\hat u|^2_{H^{m+1}(\hat T_{(n)})}\leq C\sum_{|\alpha_\perp|+\alpha_3\leq m+1}\|\rho_{\hat e}^{|\alpha_\perp|-1}\pa^{\alpha_\perp}\pa_{\hat z}^{\alpha_3}\hat u\|^2_{L^2(\hat T_{(n)})}.\end{eqnarray} 
Define $\hat w:=\hat u-\chi\hat u$. Then, by  the usual interpolation error estimate, we have
\begin{eqnarray}
    |\hat u - \hat u_{I} |_{H^1
    (\hat T_{(n)})}& =&  |\hat w+\chi \hat u - \hat u_{I} |_{H^1 (\hat T_{(n)})}  \leq  |\hat w|_{H^1
    (\hat T_{(n)})}+|\chi\hat u-\hat u_{I}|_{H^1
    (\hat T_{(n)})}\nonumber\\
  &=& |\hat w|_{H^1 (\hat T_{(n)})}+|\chi\hat u -
  (\chi\hat u)_{I}|_{H^1 (\hat T_{(n)})}\leq C(\|\hat u\|_{H^1 (\hat T_{(n)})}+|\chi\hat u|_{H^{m+1}(\hat T_{(n)})}), \label{eqn.new111111} 
\end{eqnarray}
where $C$ depends on $m$ and, through $\chi$, the nodes in the
$\hat T_{(n)}$.  In $L_{e,n}$, $\rho_e(x, y, z)=\kappa_e^{n-1}\rho_{\hat e}(\hat x, \hat y, \hat z)$. Therefore, by (\ref{eqn.etr2}), (\ref{eqn.new111111}), (\ref{eqn.mm}), and (\ref{eqn.ke1}), we have
 \begin{eqnarray*}
\|\partial_x(u-u_I)\|_{L^2(T_{(n)})}^2&\leq& C2^{1-k}\kappa_e^{n-k}\big(\|\partial_{\hat x}(\hat u-\hat u_I)\|_{L^2(\hat T_{(n)})}^2+\|\partial_{\hat z}(\hat u-\hat u_I)\|^2_{L^2(\hat T_{(n)})}\big)\\
&\leq& C2^{1-k}\kappa_e^{n-k}\sum_{|\alpha_\perp|+\alpha_3\leq m+1}\|\rho_{\hat e}^{|\alpha_\perp|-1}\pa^{\alpha_\perp}\pa_{\hat z}^{\alpha_3}\hat u\|^2_{L^2(\hat T_{(n)})} \\
&\leq& C\sum_{|\alpha_\perp|+\alpha_3\leq m+1}2^{2(1-k)\alpha_3}\kappa_e^{2(n-k)\alpha_3}\|\rho_e^{|\alpha_\perp|-1}\partial^{\alpha_\perp}\partial_z^{\alpha_3}u\|^2_{L^2(T_{(n)})}\\
&\leq &C\kappa_{e}^{2na_{e}}\|u\|_{\mathcal M^{m+1}_{\ba+\bone}(T_{(n)})}^{2} \leq Ch^{2m}\|u\|_{\mathcal M^{m+1}_{\ba+\bone}(T_{(n)})}^2.\nonumber
\end{eqnarray*}
A similar calculation for the derivative with respect to $y$ gives
\ben
\|\partial_y(u-u_I)\|_{L^2(T_{(n)})}\leq Ch^{m}\|u\|_{\mathcal M^{m+1}_{\ba+\bone}(T_{(n)})}.
\een
In the $z$-direction, by (\ref{eqn.etr2}), (\ref{eqn.new111111}), (\ref{eqn.mm}),  and (\ref{eqn.ke1}), we have
\ben
\|\partial_z(u-u_I)\|_{L^2(T_{(n)})}^2&=& (2^{1-k}\kappa_e^{n-k})\kappa_e^{2(n-1)}(2^{k-1}\kappa_e^{k-n})^2\|\partial_{\hat z}(\hat u-\hat u_I)\|^2_{L^2(\hat T_{(n)})}\\
&\leq& C(2^{1-k}\kappa_e^{n-k})\kappa_e^{2(n-1)}(2^{k-1}\kappa_e^{k-n})^2\sum_{|\alpha_\perp|+\alpha_3\leq m+1}\|\rho_{\hat e}^{|\alpha_\perp|-1}\pa^{\alpha_\perp}\pa_{\hat z}^{\alpha_3}\hat u\|^2_{L^2(\hat T_{(n)})}  \\
&\leq& C\sum_{|\alpha_\perp|+\alpha_3\leq m+1}(2^{1-k}\kappa_e^{n-k})^{2\alpha_3}(2^{k-1}\kappa_e^{k})^2\|\rho_e^{|\alpha_\perp|-1}\partial^{\alpha_\perp}\partial_z^{\alpha_3}u\|^2_{L^2(T^{(n)})}\\
&\leq &C\kappa_{e}^{2na_{e}}\|u\|_{\mathcal M^{m+1}_{\ba+\bone}(T_{(n)})}^{2} \leq Ch^{2m}\|u\|_{\mathcal M^{m+1}_{\ba+\bone}(T_{(n)})}^2.\nonumber
\een
Thus, we have proved the estimate for Case II.

Hence, the corollary is proved by summing up the estimates in Theorem \ref{thm.11} and the estimates for all the tetrahedra $T_{(n)}$ in $L_{e, n}$.
\end{proof}

\subsection{Estimates on initial $ev$-tetrahedra in $\maT_0$} In this subsection, we denote by $T_{(0)}=\triangle^4x_0x_1x_2x_3\in\maT_0$ an $ev$-tetrahedron, such that $x_0=v\in\maV$ and $x_0x_1$ is on the edge $e\in\maE$. Then, we first define mesh layers associated with $\maT_n$ on $T_{(0)}$.

\begin{definition} \label{def.evlayer}(Mesh Layers in $ev$-tetrahedra)  For $1\leq i\leq n$, the $i$th refinement on $T_{(0)}$ produces a small tetrahedron with $x_{0}$ as a vertex. We denote by $P_{ev, i}$ the face of this small tetrahedron whose closure does not contain $x_0$ (see the last two pictures in Figure {\ref{fig.n3}}). Then, for the mesh $\maT_n$ on $T_{(0)}$, we define the $i$th mesh layer $L_{ev, i}$, $1\leq i<n$, as the region in $T_{(0)}$ between $P_{ev,i}$ and $P_{ev, i+1}$. We denote by $L_{ev, 0}$ the region in  $T_{(0)}$ between $\triangle^3x_1x_2x_3$ and $P_{ev, 1}$ and let $L_{ev, n}\subset T_{(0)}$ be the small tetrahedron with $x_0$ as a vertex that is  generated in the $n$th refinement.
\end{definition}

For each $ev$-tetrahedron, one extra refinement results in one $ev$-tetrahedron, one $e$-tetrahedron, two $v_e$-tetrahedra, and four $o$-tetrahedra. Let $T=\triangle^4\gamma_0\gamma_1\gamma_2\gamma_3\subset T_{(0)}$ be an $ev$-tetrahedron generated by some subsequent refinements of $T_{(0)}$, with $\gamma_0=x_0$ and $\gamma_0\gamma_1$  on the edge $e\in\maE$.  We  define the relative distances $c_{\gamma, 1}$ and $c_{\gamma, 2}$ for $T$ using the same notation as in Definition \ref{def.rdis} (see also Remark \ref{rk.cp1}). In the next lemma, we show the analogue of Lemma \ref{lem.451} for  $ev$-tetrahedra. Namely, the relative  distances are bounded   for $ev$-tetrahedra with respect to the refinement.

\begin{lemma}\label{lem.new451} Let $T=\triangle^4\gamma_0\gamma_1\gamma_2\gamma_3\subset T_{(0)}$ be an $ev$-tetrahedron in $\maT_i$, $1\leq i<n$, with $\gamma_0=x_0$ and $\gamma_0\gamma_1$  on the edge $e\in\maE$.  Let $T_R\subset T$ be the $ev$-tetrahedron  in $\maT_{i+1}$. Denote by $c_T$ and $c_R$ the absolute distances (\ref{eqn.adis}) for $T$ and $T_R$, respectively. Then, $c_R\leq \max(c_T, 1)$.
\end{lemma}
\begin{proof} Recall the grading parameters $\kappa_v$, $\kappa_e$, and $\kappa_{ev}$ for $T_{(0)}$ with  $\kappa_v, \kappa_e\geq \kappa_{ev}$. We use Figure \ref{fig.441} to demonstrate the proof. Then, $T_R=\triangle^4\gamma_0\gamma_4\gamma_5\gamma_6$. Consider the triangles on the face $\triangle^3\gamma_0\gamma_1\gamma_2$ of $T$, induced by the sub-tetrahedra after one refinement on $T$, where $\gamma_5'$ and $\gamma_2'$ are the orthogonal projections of $\gamma_5$ and $\gamma_2$ on the singular edge. However, note that instead of the mid-point of $\gamma_0\gamma_1$ for the $e$-tetrahedron, the location of $\gamma_4$ here is given by $|\gamma_0\gamma_4|=\kappa_v|\gamma_0\gamma_1|$ for the $ev$-tetrahedron. Let $c_{\gamma_2, 1}, c_{\gamma_2, 2}$ (resp. $c^R_{\gamma_5, 1}, c^R_{\gamma_5, 2}$) be the relative distances  of $\gamma_2$ in $T$ (resp. $\gamma_5$ in $T_R$).  

Based on Algorithm \ref{alg.n1}, $|\gamma_0\gamma_5'|=\kappa_{ev}|\gamma_0\gamma_2'|$. By (\ref{eqn.cg1}),  $c_{\gamma_5, 1}^R$ and $c_{\gamma_2, 1}$ have the same sign. Then, we first show $|c_{\gamma_5, 1}^R|, |c_{\gamma_5, 2}^R|\leq \max(|c_{\gamma_2, 1}|, |c_{\gamma_2, 2}|,1)$ by considering the following cases, in which the calculations are based on the definitions in (\ref{eqn.cg1}) and (\ref{eqn.cg2}).

If $ c_{\gamma_2,1}<0$, we have
$$
0> c^R_{\gamma_5,1}=-|{\gamma_0\gamma_5'}|/|{\gamma_0\gamma_4}|=-\kappa^{-1}_v\kappa_{e,v}|{\gamma_{0}\gamma_2'}|/|{\gamma_{0}\gamma_1}|=\kappa^{-1}_v\kappa_{ev} c_{\gamma_2, 1}\geq c_{\gamma_2,1}.
$$
Therefore, $|c^R_{\gamma_5,1}|\leq |c_{\gamma_2,1}|$.
Meanwhile, we have $$1\leq c^R_{\gamma_5,2}=1-c^R_{\gamma_5,1}=1-\kappa^{-1}_v\kappa_{ev} c_{\gamma_2,1}<1-c_{\gamma_2,1}=c_{\gamma_2,2}.$$
Therefore, $|c^R_{\gamma_5,2}|\leq |c_{\gamma_2,2}|$.


If $0\leq c_{\gamma_2,1}<\kappa_v\kappa_{ev}^{-1}$, we have $c^R_{\gamma_5,1}\geq 0$ and
$$
c^R_{\gamma_5,1}=|{\gamma_0\gamma_5'}|/|{\gamma_0\gamma_4}|=\kappa^{-1}_v\kappa_{ev}|{\gamma_0\gamma_2'}|/|{\gamma_0\gamma_4}|=\kappa^{-1}_v\kappa_{ev} c_{\gamma_2,1}<1.
$$
Meanwhile, we have $$0\leq c^R_{\gamma_5,2}=1-c^R_{\gamma_5,1}\leq  1.$$

If $ c_{\gamma_2,1}\geq \kappa_v\kappa_{ev}^{-1}$, we have
$$
1\leq c^R_{\gamma_5,1}=|{\gamma_0\gamma_5'}|/|{\gamma_0\gamma_4}|=\kappa^{-1}_v\kappa_{ev}|{\gamma_0\gamma_2'}|/|{\gamma_0\gamma_4}|=\kappa^{-1}_v\kappa_{ev} c_{\gamma_2,1}\leq |c_{\gamma_2,1}|.
$$
Meanwhile, we have $$0\geq c^R_{\gamma_5,2} =1-c^R_{\gamma_5,1} =1-\kappa^{-1}_v\kappa_{ev} c_{\gamma_2,1}\geq 1-c_{\gamma_2,1}=c_{\gamma_2, 2}.$$
Therefore, $|c^R_{\gamma_5,2}|\leq |c_{\gamma_2,2}|$.

Hence, $|c^R_{\gamma_5,1}|, |c^R_{\gamma_5,2}|\leq \max(|c_{\gamma_2,1}|, |c_{\gamma_5,2}|,1)$. Using a similar calculation, we can also obtain the same estimate for relative distances of $\gamma_3$ and $\gamma_6$. Then, the proof is completed by combining these estimates and by the definition of the absolute distance (\ref{eqn.adis}). 
\end{proof}

Now, we define the reference element for the $ev$-tetrahedron.

\begin{definition} \label{def.evrt}(The Reference $ev$-tetrahedron) 
We shall use the  tetrahedron $\hat T=\triangle^4\hat x_0\hat x_1\hat x_2\hat x_3$ in Definition \ref{def.rt} as our reference element in this subsection. 
For $T_{(0)}$, recall the grading parameters  $\kappa_v$ and $\kappa_e$ associated with $x_0$ and $x_0x_1$, respectively. For the reference $ev$-tetrahedron $\hat T$, one graded refinement using the same parameters $\kappa_v, \kappa_e$, and $\kappa_{ev}$ for $\hat x_0$ and $\hat x_0\hat x_1$ gives rise to a triangulation on $\hat T$, which we denote by $\hat \maT_1$. Define the union of the seven tetrahedra in $\hat\maT_1$ away from $\hat x_0$ to be the mesh layer $\hat L$ on $\hat T$. We denote by   $\hat \maL$ the initial triangulation of $\hat L$ that contains these seven tetrahedra. 
\end{definition}

Then, we construct a mapping from an $ev$-tetrahedron $T\subset T_{(0)}$ in $\maT_i$ to $\hat T$. 

\begin{lemma}\label{lem.nnnew}
For an $ev$-tetrahedron $T:=\triangle^4\gamma_0\gamma_1\gamma_2\gamma_3\subset T_{(0)}$ in $\maT_i$, $0\leq i\leq n$, suppose $\gamma_0=v\in\maV$ and $\gamma_0\gamma_1\subset e\in\maE$. Use a local Cartesian coordinate system, such that $(\gamma_0+\gamma_1)/2$ is the origin, $\gamma_1$ is in the positive $z$-axis,  and $\gamma_2$ is in the $xz$-plane.  Then, there is a mapping
\begin{eqnarray}\label{eqn.mapt1111}
\mathbf B_{ev, i}= \begin{pmatrix}
   \kappa^{-i}_{ev}   &   0 & 0 \\
  0    &   \kappa^{-i}_{ev} & 0\\
  b_1 \kappa^{-i}_{ev}\ & b_2 \kappa^{-i}_{ev} &  \kappa_v^{-i}
\end{pmatrix}
\end{eqnarray}
with $|b_1|,|b_2|\leq C_0$, for $C_0\geq 0$ depending on $T_{(0)}$, 
such that  $\mathbf B_{ev, i}: T\rightarrow \hat T$ is a bijection. 
\end{lemma}
\begin{proof} Recall $\hat\lambda_k$ and $\hat\xi_k$, $k=2,3$, in Definition \ref{def.rt}. 
Based on Algorithm \ref{alg.n1}, we have $\gamma_2=(\kappa^{i}_{ev}\hat\lambda_2, 0, \zeta_2)$,  $\gamma_3=(\kappa^{i}_{ev}\hat\lambda_3,\kappa^{i}_{ev}\hat\xi_3, \zeta_3$), and   $|{\gamma_0\gamma_1}|=\kappa_v^{i}l_0$=$\kappa_v^{i}|x_0x_1|$, where $\zeta_2$ and $\zeta_3$ are the $z$-coordinates of the vertices $\gamma_2$ and $\gamma_3$, respectively.  Then, the transformation 
\begin{eqnarray*}
\mathbf A_1:= \begin{pmatrix}
   \kappa^{-i}_{ev}   &   0 & 0 \\
  0    &   \kappa^{-i}_{ev} & 0\\
  0 & 0 & \kappa_v^{-i}
\end{pmatrix}
\end{eqnarray*}
maps $T$ to a tetrahedron, with vertices $\mathbf A_1\gamma_0=(0, 0, -l_0/2)$,  $\mathbf A_1\gamma_1=(0, 0, l_0/2)$, $\mathbf A_1\gamma_2=(\hat\lambda_2, 0, \kappa_v^{-i}\zeta_2)$, and $\mathbf A_1\gamma_3=(\hat\lambda_3, \hat\xi_3, \kappa_v^{-i}\zeta_3)$. Now, let 
$$b_1=-(\kappa_v^{-i}\zeta_2+l_0/2)/\hat\lambda_2, \qquad b_2=[\kappa_v^{-i}(\zeta_2\hat\lambda_3-\hat\lambda_2\zeta_3)+l_0(\hat\lambda_3-\hat\lambda_2)/2]/\hat\lambda_2\hat\xi_3.$$
Define 
\begin{eqnarray*}
\mathbf A_2:= \begin{pmatrix}
   1   &   0 & 0 \\
  0    &   1 & 0\\
  b_1 & b_2 & 1
\end{pmatrix}.
\end{eqnarray*}
By Lemma \ref{lem.new451}, the absolute distance for $T$ is bounded by a constant determined by $T_{(0)}$. Therefore,  we have $|\zeta_2|, |\zeta_3|\leq C|{\gamma_0\gamma_1}|=C\kappa_v^{i}l_0$, where $C$ depends on the shape regularity of $T_{(0)}$. In addition, since $\hat\lambda_2, \hat\lambda_3, \hat\xi_3$ all depend on the shape regularity of $T^{(0)}$ and $\hat\lambda_2,  \hat\xi_3\neq 0$, the transformation
\begin{eqnarray*}
\mathbf B_{ev, i}:=\mathbf A_2\mathbf A_1= \begin{pmatrix}
   \kappa^{-i}_{ev}   &   0 & 0 \\
  0    &   \kappa^{-i}_{ev} & 0\\
  b_1 \kappa^{-i}_{ev}\ & b_2 \kappa^{-i}_{ev} & \kappa_v^{-i}
\end{pmatrix}
\end{eqnarray*}
maps $T$ to $\hat T$ with  $|b_1|,|b_2|\leq C_0$, where $C_0\geq 0$ depends on $T_{(0)}$ but not on $i$. This completes the proof.
 \end{proof}

Using the mapping in Lemma \ref{lem.nnnew}, we present the interpolation error estimate in the mesh layer $L_{ev, i}$.

\begin{theorem}\label{lem.eve} Let $T_{(0)}=\triangle^4x_0x_1x_2x_3\in\maT_0$ be an $ev$-tetrahedron defined above. Let   $L_{ev, i}$ be the mesh layer in  Definition \ref{def.evlayer}, $0\leq i<n$. Recall the parameters $a_v$, $a_e$, and $a_{ev}$ associated to $\kappa_v$, $\kappa_e$, and $\kappa_{ev}$ in (\ref{eqn.kv1}) -- (\ref{eqn.kev1}). Define 
\be\label{eqn.av}a_V:= (m+1)(1-a^{-1}_va_{ev})+a_{ev}.\ee
Suppose 
$$\sum_{|\alpha_\perp|+\alpha_3\leq m+1}\|\rho_v^{\alpha_3-a_V+a_{e}}\rho_{e}^{|\alpha_\perp |-1-a_e}\partial^{\alpha_\perp}\partial_z^{\alpha_3}u\|^2_{L^2(T_{(0)})}<\infty.$$ Let  $u_I$ be the nodal interpolation on $\maT_n$. Then, we have 
$$
|u-u_I|^2_{H^1(L_{ev, i})}\leq Ch^{2m}\sum_{|\alpha_\perp|+\alpha_3\leq m+1}\|\rho_v^{\alpha_3-a_V+a_{e}}\rho_{e}^{|\alpha_\perp |-1-a_e}\partial^{\alpha_\perp}\partial_z^{\alpha_3}u\|^2_{L^2(L_{ev,i})},
$$
where $h=2^{-n}$ and $C$ depends on $T_{(0)}$ and $m$.
\end{theorem}

\begin{proof} Let $T_{(i)}\subset T_{(0)}$ be the $ev$-tetrahedron in $\maT_i$. Then by Definition \ref{def.evlayer}, we have $L_{ev, i}=T_{(i)}\setminus T_{(i+1)}$.  Then, the mapping  $\mathbf B_{ev,i}$ in (\ref{eqn.mapt1111}) translates $L_{ev,i}$ to $\hat{L}$ (see Definition \ref{def.evrt}). For a point $(x,y,z)\in L_{ev,i}$, let  $(\hat x, \hat y, \hat z)\in \hat L$ be its image under $\mathbf B_{ev,i}$. For a function $v$ on $L_{ev, i}$, define the function $\hat v$ on $\hat L$ by 
$$
\hat v(\hat x, \hat y, \hat z) :=v(x, y, z).
$$
Let $\rho_{\hat e}$ be the distance to $\hat x_0\hat x_1$ on the reference tetrahedron $\hat T$.   Then, it is clear that $\rho_{e}(x, y, z)=\kappa_{ev}^{i}\rho_{\hat e}(\hat x, \hat y, \hat z)$ on $L_{ev,i}$.  Meanwhile, $\mathbf B_{ev, i}$ maps the triangulation $\maT_n$ on $L_{ev, i}$ to a graded triangulation  on $\hat L$ that is obtained after $i+1-n$ refinements of the initial mesh $\hat \maL$. Note that the subsequent refinements on $\hat\maL$ are anisotropic with the parameter $\kappa_{e}$ toward $\hat x_0\hat x_1$, since $\hat\maL$ does not contain $ev$- or $v$-tetrahedra.

Then, by the mapping  (\ref{eqn.mapt1111}),  the scaling argument, Corollary \ref{cor.ete}, (\ref{eqn.rhoe}), and (\ref{eqn.ke1}), we have
 \begin{eqnarray}
\|\partial_x(u-u_I)\|_{L^2(L_{ev,i})}^2&\leq& C\kappa_v^{i}\big(\|\partial_{\hat x}(\hat u-\hat u_I)\|_{L^2(\hat L)}^2+\|\partial_{\hat z}(\hat u-\hat u_I)\|^2_{L^2(\hat L)}\big)\nonumber\\
&\leq& C\kappa_v^{i}2^{2m(i-n)}\sum_{|\alpha_\perp|+\alpha_3\leq m+1}\|\rho_{\hat{e}}^{|\alpha_\perp|-1-a_e}\partial^{\alpha_\perp}\partial_{\hat{z}}^{\alpha_3}\hat u\|^2_{L^2(\hat L)} \nonumber \\
&\leq& C2^{2m(i-n)}\sum_{|\alpha_\perp|+\alpha_3\leq m+1}\kappa_v^{2i\alpha_3}\kappa_{ev}^{2ia_e}\|\rho_e^{|\alpha_\perp|-1-a_e}\partial^{\alpha_\perp}\partial_z^{\alpha_3}u\|^2_{L^2(L_{ev,i})}\nonumber\\
&\leq& C2^{-2mn}\sum_{|\alpha_\perp|+\alpha_3\leq m+1}2^{2im}\kappa_v^{2i\alpha_3}\kappa_{ev}^{2ia_e}\|\rho_{e}^{|\alpha_\perp |-1-a_e}\partial^{\alpha_\perp}\partial_z^{\alpha_3}u\|^2_{L^2(L_{ev,i})}.\label{eqn.av-1}
\end{eqnarray}
Note that $\kappa_{ev}^{i}\lesssim\rho_v\lesssim\kappa_v^i$ on $L_{ev, i}$,  $a_v, a_e\geq a_{ev}$ (see (\ref{eqn.kev1})) and $a_V\geq a_v$.  Then, we consider all the possible  cases below.\\
(I) ($\alpha_3\leq a_v$.) Then,  we have
\ben
\kappa_v^{i(\alpha_3-a_v)}\lesssim \rho_v^{\alpha_3-a_v}.
\een
Then, by (\ref{eqn.kv1}), we have
\be\label{eqn.av0}
2^{im}\kappa_v^{i\alpha_3}\kappa_{ev}^{ia_{e}}\lesssim \rho_v^{\alpha_3-a_v}\kappa_{ev}^{ia_{e}}\lesssim \rho_v^{\alpha_3-a_v+a_{e}}.
\ee
(II) ($(1-a^{-1}_va_{ev})(m+1)<\alpha_3\leq m+1$.) 
Note that $0<a_v\leq m$. Therefore, by (\ref{eqn.kev1}), we have
\ben
\kappa_v^{i\alpha_3}=2^{-im\alpha_3/a_v}\leq 2^{-im\alpha_3/a_{ev}+ima_V/a_{ev}-im}=\kappa_{ev}^{i(\alpha_3-a_V+a_{ev})}.
\een
Note that $\alpha_3-a_V+a_{ev}>0$, therefore, 
\be\label{eqn.av1}
2^{im}\kappa_v^{i\alpha_3}\kappa_{ev}^{ia_{e}}\leq 2^{im}\kappa_{ev}^{i(\alpha_3-a_V+a_{ev})}\kappa_{ev}^{ia_{e}}\lesssim \rho_v^{\alpha_3-a_V+a_{e}}.
\ee
(III) ($a_v<\alpha_3\leq (1-a_v^{-1}a_{ev})(m+1)$.) If $a_{ev}=a_v$, we have $(1-a_v^{-1}a_{ev})(m+1)=0$, and therefore such $\alpha_3$ does not exist. Thus, we only need to consider the case $a_{ev}<a_v$. Note that $\alpha_3-a_V+a_{ev}\leq 0$ and $a_{ev}\leq a_e$.  Therefore, by (\ref{eqn.kev1}), we have
\be\label{eqn.av2}
2^{im}\kappa_v^{i\alpha_3}\kappa_{ev}^{ia_{e}}= \kappa_v^{i\alpha_3}\kappa_{ev}^{i(a_e-a_{ev})}  \lesssim \rho_v^{\alpha_3-a_V+a_{ev}}\rho_v^{(a_e-a_{ev})}=\rho_v^{\alpha_3-a_V+a_{e}}.
\ee
Therefore, choosing $a_V$ as in (\ref{eqn.av}), by (\ref{eqn.av-1}) --  (\ref{eqn.av2}), we have shown that 
\be\label{eqn.xxx}
\|\partial_x(u-u_I)\|_{L^2(L_{ev,i})}^2&\leq&C2^{-2mn}\sum_{|\alpha_\perp|+\alpha_3\leq m+1}\|\rho_v^{\alpha_3-a_V+a_{e}}\rho_{e}^{|\alpha_\perp |-1-a_e}\partial^{\alpha_\perp}\partial_z^{\alpha_3}u\|^2_{L^2(L_{ev,i})}.\label{eqn.av-2}
\ee

In the $y$-direction, with a similar process, we obtain
\be
\|\partial_y(u-u_I)\|_{L^2(L_{ev,i})}^2&\leq&C2^{-2mn}\sum_{|\alpha_\perp|+\alpha_3\leq m+1}\|\rho_v^{\alpha_3-a_V+a_{e}}\rho_{e}^{|\alpha_\perp |-1-a_e}\partial^{\alpha_\perp}\partial_z^{\alpha_3}u\|^2_{L^2(L_{ev,i})}.\label{eqn.av-2}
\ee

In the $z$-direction, by the mapping  (\ref{eqn.mapt1111}),  the scaling argument, and Corollary \ref{cor.ete},  we have
\be
\|\partial_z(u-u_I)\|_{L^2(L_{ev, i})}^2&=& \kappa_v^{-i}\kappa_{ev}^{2i}\|\partial_{\hat z}(\hat u-\hat u_I)\|^2_{L^2(\hat L)}\nonumber\\
&\leq& C2^{2m(i-n)}\kappa_v^{-i}\kappa_{ev}^{2i}\sum_{|\alpha_\perp|+\alpha_3\leq m+1}\| \rho_{\hat e}^{|\alpha_\perp|-1-a_e}\partial^{\alpha_\perp}\partial_{\hat z}^{\alpha_3} \hat u\|^2_{L^2(\hat L)} \nonumber \\
&\leq& C2^{-2mn}\sum_{|\alpha_\perp|+\alpha_3\leq m+1}2^{{2im}}\kappa_v^{2i(\alpha_3-1)}\kappa_{ev}^{2i(1+a_e)}\|\rho_{e}^{|\alpha_\perp|-1-a_e}\partial^{\alpha_\perp}\partial_z^{\alpha_3}u\|^2_{L^2(L_{ev, i})}\nonumber\\
&\leq& C2^{-2mn}\sum_{|\alpha_\perp|+\alpha_3\leq m+1}2^{{2im}}\kappa_v^{2i\alpha_3}\kappa_{ev}^{2ia_e}\|\rho_{e}^{|\alpha_\perp|-1-a_e}\partial^{\alpha_\perp}\partial_z^{\alpha_3}u\|^2_{L^2(L_{ev, i})}\nonumber\\
&\leq& C2^{-2mn}\sum_{|\alpha_\perp|+\alpha_3\leq m+1}\|\rho_v^{\alpha_3-a_V+a_e}\rho_{e}^{|\alpha_\perp |-1-a}\partial^{\alpha_\perp}\partial_z^{\alpha_3}u\|^2_{L^2(L_{ev,i})},\label{eqn.zzz}
\ee
where the last inequality follows from the analysis in (\ref{eqn.av0}) -- (\ref{eqn.av2}).

Hence, the proof is completed by the estimates in (\ref{eqn.xxx}) -- (\ref{eqn.zzz}).
\end{proof}

Then, we are ready to obtain the interpolation error estimate on the entire $ev$-tetrahedron $T_{(0)}$.

\begin{corollary}\label{cor.evt} Let $T_{(0)}\in\maT_0$ be an $ev$-tetrahedron  as in Theorem \ref{lem.eve}. Recall $a_V$ from (\ref{eqn.av}). Suppose 
$$\sum_{|\alpha_\perp|+\alpha_3\leq m+1}\|\rho_v^{\alpha_3-a_V+a_{e}}\rho_{e}^{|\alpha_\perp |-1-a_e}\partial^{\alpha_\perp}\partial_z^{\alpha_3}u\|^2_{L^2(T_{(0)})}<\infty.$$ Let  $u_I$ be the nodal interpolation on $\maT_n$. Then, we have 
$$
|u-u_I|^2_{H^1(T_{(0)})}\leq Ch^{2m}\sum_{|\alpha_\perp|+\alpha_3\leq m+1}\|\rho_v^{\alpha_3-a_V+a_{e}}\rho_{e}^{|\alpha_\perp |-1-a_e}\partial^{\alpha_\perp}\partial_z^{\alpha_3}u\|^2_{L^2(T_{(0)})},
$$
where $h=2^{-n}$ and $C$ depends on $T_{(0)}$ and $m$.\end{corollary}
\begin{proof}
By Theorem \ref{lem.eve}, it suffices to show 
$$
|u-u_I|^2_{H^1(L_{ev,n})}\leq C2^{-2mn}\sum_{|\alpha_\perp|+\alpha_3\leq m+1}\|\rho_v^{\alpha_3-a_V+a_{e}}\rho_{e}^{|\alpha_\perp |-1-a_e}\partial^{\alpha_\perp}\partial_z^{\alpha_3}u\|^2_{L^2(L_{ev,n})}.
$$

By Lemma \ref{lem.nnnew},  $\mathbf B_{ev, n}(L_{ev,n})=\hat T$. For $(x,y , z)\in L_{ev,n}$, let $(\hat x, \hat y, \hat z)\in \hat T$ be its image under $\mathbf B_{ev,n}$. For a function $v$ on $L_{ev,n}$, we define $\hat v$ on $\hat T$ by
$$
\hat v(\hat x, \hat y, \hat z):=v(x, y, z).
$$
Now, let $\chi$ be a smooth cutoff function on $\hat T$ such that
$\chi=0$ in a neighborhood of the edge $\hat e:=\hat x_0\hat x_1$ and $=1$ at every other node of
$\hat T$. Let $\rho_{\hat{v}}$ be the distance from $(\hat x, \hat y, \hat z)$ to $\hat x_0$. Then, by (\ref{eqn.mapt1111}),
\be\label{eqn.hatv}\kappa^n_{ev}\rho_{\hat{v}}(\hat x,\hat y, \hat z)\lesssim\rho_v(x,y,z)\lesssim \kappa^n_v\rho_{\hat{v}}(\hat x,\hat y, \hat z),
\ee  and $\kappa^n_{ev}\rho_{\hat e}(\hat x,\hat y,\hat z)=\rho_{e}(x, y, z)$. Let $\hat u_I$ be the interpolation of $\hat u$ on the  reference tetrahedron $\hat T$.  Since $\chi\hat u=0$ in the neighborhood of $\hat e$, $(\chi\hat u)_{I}=\hat u_I$ and 
\be\label{eqn.equ1}
|\chi\hat u|_{H^{m+1}(\hat T)}\leq C\sum_{|\alpha_\perp|+\alpha_3\leq m+1}\|\rho_{\hat e}^{|\alpha_\perp|-1-a_e}\rho_{\hat v}^{\alpha_3-a_V+a_{e}}\pa^{\alpha_\perp}\pa_{\hat z}^{\alpha_3}\hat u\|^2_{L^2(\hat T)}.
\ee 
Note that by (\ref{eqn.av}), $a_V\geq a_{ev}$. Define $\hat w:=\hat u-\chi\hat u$.   Then, by  the usual interpolation error estimate,  $\rho_{\hat{e}}\lesssim \rho_{\hat{v}}$, and (\ref{eqn.equ1}), we have
\begin{eqnarray}
    |\hat u - \hat u_{I} |_{H^1
    (\hat T)}& =&  |\hat w+\chi \hat u - \hat u_{I} |_{H^1 (\hat T)}  \leq  |\hat w|_{H^1
    (\hat T)}+|\chi\hat u-\hat u_{I}|_{H^1
    (\hat T)}\nonumber\\
  &=& |\hat w|_{H^1 (\hat T)}+|\chi\hat u -
  (\chi\hat u)_{I}|_{H^1 (\hat T)}\leq C(\|\hat u\|_{H^1 (\hat T)}+|\chi\hat u|_{H^{m+1}(\hat T)}),\nonumber\\
  &\leq& C\sum_{|\alpha_\perp|+\alpha_3\leq m+1}\|\rho_{\hat e}^{|\alpha_\perp|-1-a_e}\rho_{\hat v}^{\alpha_3-a_V+a_{e}}\pa^{\alpha_\perp}\pa_{\hat z}^{\alpha_3}\hat u\|^2_{L^2(\hat T)}, \label{eqn.new11111}
\end{eqnarray}
where $C$ depends on $m$ and, through $\chi$, the nodes on $\hat T$. Then, using (\ref{eqn.new11111}), the scaling argument based on (\ref{eqn.mapt1111}), and the relation $\rho_{\hat e}(\hat x, \hat y,\hat z)=\kappa_{ev}^{-n}\rho_e(x, y, z)$, we have
 \begin{eqnarray}
\|\partial_x(u-u_I)\|_{L^2(L_{ev,n})}^2&\leq& C\kappa_v^{n}\big(\|\partial_{\hat x}(\hat u-\hat u_I)\|_{L^2(\hat T)}^2+\|\partial_{\hat z}(\hat u-\hat u_I)\|^2_{L^2(\hat T)}\big)\nonumber\\
&\leq& C\kappa_v^{n}\sum_{|\alpha_\perp|+\alpha_3\leq m+1}\|\rho_{\hat v}^{\alpha_3-a_V+a_{e}}\rho_{\hat e}^{|\alpha_\perp|-1-a_e}\pa^{\alpha_\perp}\pa^{\alpha_3}_{\hat z}\hat u\|_{L^2(\hat T)}^2  \nonumber\\
&\leq& C\sum_{|\alpha_\perp|+\alpha_3\leq m+1}\kappa_v^{2n\alpha_3}\kappa_{ev}^{2na_e}\|\rho_{\hat v}^{\alpha_3-a_V+a_{e}}\rho_{e}^{|\alpha_\perp|-1-a_e}\pa^{\alpha_\perp}\pa^{\alpha_3}_{z} u\|_{L^2(L_{ev,n})}^2 \nonumber \\
&\leq& C 2^{-2mn}\sum_{|\alpha_\perp|+\alpha_3\leq m+1}2^{2nm}\kappa_v^{2n\alpha_3}\kappa_{ev}^{2na_e}\|\rho_{\hat v}^{\alpha_3-a_V+a_{e}}\rho_{e}^{|\alpha_\perp|-1-a_e}\pa^{\alpha_\perp}\pa^{\alpha_3}_{z} u\|_{L^2(L_{ev,n})}^2. \label{eqn.rhov0}
\end{eqnarray}
Then, we consider the following cases. \\
(I) ($\alpha_3\leq a_v$.) By (\ref{eqn.kv1}), (\ref{eqn.hatv}), $a_V\geq a_v$, and $\alpha_3-a_v\leq 0$, we have 
\be\label{eqn.rhov1}
2^{nm}\kappa_v^{n\alpha_3}\kappa_{ev}^{na_e}\rho_{\hat v}^{\alpha_3-a_V+a_e}&=&\kappa_v^{n(\alpha_3-a_v)}\rho_{\hat v}^{\alpha_3-a_v}\kappa_{ev}^{na_e}\rho_{\hat v}^{a_e}\rho_{\hat v}^{a_v-a_V}\nonumber\\
&\lesssim& \rho_{v}^{\alpha_3-a_v+a_e}\rho_{\hat v}^{a_v-a_V}\lesssim \rho_{v}^{\alpha_3-a_V+a_e}.
\ee
(II) ($(1-a^{-1}_va_{ev})(m+1)<\alpha_3\leq m+1$.) Following  the calculation in (\ref{eqn.av1}), by (\ref{eqn.kev1}) and (\ref{eqn.hatv}), we have 
\be\label{eqn.rhov2}
2^{nm}\kappa_v^{n\alpha_3}\kappa_{ev}^{na_{e}}\rho_{\hat v}^{\alpha_3-a_V+a_e}&\leq& 2^{nm}\kappa_{ev}^{n(\alpha_3-a_V+a_{ev})}\kappa_{ev}^{na_{e}}\rho_{\hat v}^{\alpha_3-a_V+a_e}\nonumber\\
&=&\kappa_{ev}^{n(\alpha_3-a_V)}\kappa_{ev}^{na_{e}}\rho_{\hat v}^{\alpha_3-a_V+a_e}\lesssim \rho_v^{\alpha_3-a_V+a_{e}}.
\ee
(III) ($a_v<\alpha_3\leq (1-a_v^{-1}a_{ev})(m+1)$.) If $a_{ev}=a_v$, we have $(1-a_v^{-1}a_{ev})(m+1)=0$, and therefore such $\alpha_3$ does not exist. Thus, we only need to consider the case $a_{ev}<a_v$. Note that $\alpha_3-a_V+a_{ev}\leq 0$ and $a_{ev}\leq a_e$.  Therefore, by (\ref{eqn.kev1}) and (\ref{eqn.hatv}), we have
\be\label{eqn.rhov3}
2^{nm}\kappa_v^{n\alpha_3}\kappa_{ev}^{na_{e}}\rho_{\hat v}^{\alpha_3-a_V+a_e}&=&\kappa_v^{n\alpha_3}\kappa_{ev}^{n(a_e-a_{ev})} \rho_{\hat v}^{\alpha_3-a_V+a_{e}}\nonumber\\
&\leq& \kappa_v^{n(\alpha_3-a_V+a_{ev})} \rho_{\hat v}^{\alpha_3-a_V+a_{ev}}\kappa_{ev}^{n(a_e-a_{ev})} \rho_{\hat v}^{a_{e}-{a_{ev}}}
\lesssim \rho_v^{\alpha_3-a_V+a_{e}}.
\ee
Therefore, by (\ref{eqn.rhov0}) -- (\ref{eqn.rhov3}), we conclude 
\be\label{eqn.final1}
\|\partial_x(u-u_I)\|^2_{L^2(L_{ev,n})}\leq Ch^{2m}\sum_{|\alpha_\perp|+\alpha_3\leq m+1}\|\rho_{v}^{\alpha_3-a_V+a_{e}}\rho_{e}^{|\alpha_\perp|-1-a_e}\pa^{\alpha_\perp}\pa^{\alpha_3}_{z} u\|^2_{L^2(L_{ev,n})}.
\ee

A similar error estimate in the $y$-direction leads to
\be\label{eqn.final2}
\|\partial_y(u-u_I)\|^2_{L^2(L_{ev,n})}\leq Ch^{2m}\sum_{|\alpha_\perp|+\alpha_3\leq m+1}\|\rho_{v}^{\alpha_3-a_V+a_{e}}\rho_{e}^{|\alpha_\perp|-1-a_e}\pa^{\alpha_\perp}\pa^{\alpha_3}_{z} u\|^2_{L^2(L_{ev,n})}.
\ee

In the $z$-direction, using (\ref{eqn.new11111}), $\kappa_v\geq \kappa_{ev}$, the scaling argument based on (\ref{eqn.mapt1111}),  (\ref{eqn.ke1}), and (\ref{eqn.rhov1}) -- (\ref{eqn.rhov3}), we have
\be
\|\partial_z(u-u_I)\|_{L^2(L_{ev,n})}^2&=& \kappa_v^{-n}\kappa_{ev}^{2n}\|\partial_{\hat z}(\hat u-\hat u_I)\|^2_{L^2(\hat T)} \nonumber \\
&\leq& C\kappa_v^{-n}\kappa_{ev}^{2n}\sum_{|\alpha_\perp|+\alpha_3\leq m+1}\|\rho_{\hat v}^{\alpha_3-a_V+a_{e}}\rho_{\hat e}^{|\alpha_\perp|-1-a_e}\pa^{\alpha_\perp}\pa^{\alpha_3}_{\hat z}\hat u\|_{L^2(\hat T)}^2  \nonumber\\ 
&\leq& C\sum_{|\alpha_\perp|+\alpha_3\leq m+1}\kappa_v^{2n(\alpha_3-1)}\kappa_{ev}^{2n(1+a_e)}\|\rho_{\hat v}^{\alpha_3-a_V+a_{e}}\rho_{e}^{|\alpha_\perp|-1-a_e}\pa^{\alpha_\perp}\pa^{\alpha_3}_{z} u\|_{L^2(L_{ev,n})}^2  \nonumber\\
&\leq& C 2^{-2mn}\sum_{|\alpha_\perp|+\alpha_3\leq m+1}2^{2nm}\kappa_v^{2n\alpha_3}\kappa_{ev}^{2na_e}\|\rho_{\hat v}^{\alpha_3-a_V+a_{e}}\rho_{e}^{|\alpha_\perp|-1-a_e}\pa^{\alpha_\perp}\pa^{\alpha_3}_{z} u\|_{L^2(L_{ev,n})}^2 \nonumber \\
&\leq &Ch^{2m}\sum_{|\alpha_\perp|+\alpha_3\leq m+1}\|\rho_{v}^{\alpha_3-a_V+a_{e}}\rho_{e}^{|\alpha_\perp|-1-a_e}\pa^{\alpha_\perp}\pa^{\alpha_3}_{z} u\|_{L^2(L_{ev,n})}^2. \label{eqn.final3}
\ee

Then, the proof is completed by (\ref{eqn.final1}) -- (\ref{eqn.final3}).
\end{proof}

Then, we formulate our interpolation error analysis for the anisotropic mesh on $\Omega$.
\begin{theorem}\label{thm.mmm}
Recall $\ba$ in (\ref{eqn.aell}). Let $a_{V_T}$ be the parameter (\ref{eqn.av}) associated to the initial $ev$-tetrahedron $T\in\maT_0$. For each vertex  $v_\ell\in\maV$, let $U_\ell$ be the union of the initial $ev$-tetrahedra that have $v_\ell$ as the singular vertex. Define $\bsigma=(\sigma_1, \cdots, \sigma_{N_s})$, such that 
\ben
\sigma_\ell=\left\{\begin{array}{ll}
\max_{T\in U_\ell}(a_{V_T}),\qquad \qquad &1\leq \ell\leq N_v;\\
a_\ell, \qquad \qquad &N_v< \ell\leq N_s.
\end{array}\right.
\een
Let $\maT_n$ be the triangulation defined in Algorithm \ref{alg.n1}. For $u\in\mathcal M^{m+1}_{\bsigma+\bone}(\Omega)$, let $u_I$ be its nodal interpolation on $\maT_n$. Then, we have
$$
|u-u_I|_{H^1(\Omega)}\leq Ch^{m}\|u\|_{\mathcal M^{m+1}_{\bsigma+\bone}(\Omega)},
$$
where $h=2^{-n}$. In turn, for the finite element solution $u_n$ defined in (\ref{eqn.fems1}), we have
$$|u-u_n|_{H^1(\Omega)}\leq C \dim(S_n)^{-m/3}\|u\|_{\mathcal M^{m+1}_{\bsigma+\bone}(\Omega)},$$
where $dim(S_n)$ is the  dimension of the finite element space associated with $\maT_n$. In both estimates, the constant $C$ depends on $\maT_0$ and $m$, but not on $n$.
\end{theorem}
\begin{proof} Note that $\bsigma\geq \ba>\bzero$. Then,  the first inequality is the consequence of the definition of the weighted space $\mathcal M^{m}_{\bmu}$ and  the local interpolation error estimates on different initial tetrahedra: the $o$-tetrahedra (Lemma \ref{lem.ote}), the $v$- or $v_e$-tetrahedra (Corollary \ref{cor.vte}), the $e$-tetrahedra (Corollary \ref{cor.ete}), and the $ev$-tetrahedra (Corollary \ref{cor.evt}). 

Note that for each refinement, each tetrahedron is decomposed into 8 child tetrahedra. Therefore, the dimension of the finite element space $\dim(S_n)\sim 2^{3n}$. Thus, the second inequality follows from the best approximation property (\ref{eqn.subo111}) and  $h\sim \dim(S_n)^{-1/3}$.
\end{proof}

\begin{remark} It can be seen that for $ev$-tetrahedra, $a_V\geq a_v$. This additional regularity requirement  is needed to compensate  for the lack of the maximum angle condition in the mesh when $a_v>a_{ev}$. In the special case when $a_v=a_{ev}$,  we have $a_V=a_v$ and the new $ev$-tetrahedra generated in each refinement will satisfy the angle condition. Thus, the regularity requirement in Theorem \ref{thm.mmm} becomes $u\in\mathcal M^{m+1}_{\bsigma+\bone}(\Omega)=\mathcal M^{m+1}_{\ba+\bone}(\Omega)$.
\end{remark}

\begin{remark} Given a sufficiently smooth function $f$ in equation (\ref{eqn.n1}), the regularity of the solution $u$ (the parameters of the weighted space in the regularity estimates) depends on the geometry of the domain. See (\ref{eqn.regani}) for example. Therefore,  for a singular solution in the weighted space $\mathcal M^{m+1}_{\bsigma+\bone}(\Omega)$,   it is sufficient to choose the grading parameter $\ba$ that satisfies the condition  in Theorem \ref{thm.mmm}, in order to recover the optimal convergence rate of the finite element solution. 
\end{remark}

\section{Numerical results}\label{sec5}

In this section, using the the proposed anisotropic  finite element algorithm,  we solve the boundary value problem (\ref{eqn.n1})  on two model polyhedral domains (the prism and the Fichera corner). These domains represent typical three dimensional vertex-edge solution singularities. It will be evident that  the numerical results are align with our approximation results  
presented in Section \ref{sec4}, and thus validate our method. In both numerical tests, we use linear finite elements and let $f=1$. This is for the purpose of simplifying the demonstration of the method. High-order elements solving more complicated equations will be reported in a forthcoming paper.

\subsection{Test I} (The Prism Domain) Let $T$ be the triangle with
vertices $(0, 0), (1, 0)$, and $(0.5, 0.5)$, and let the domain be the
prism $\Omega:=\big((0, 1)^2\setminus T\big) \times(0, 1)$ (Figure \ref{fig.5}). Then, we solve equation (\ref{eqn.n1}) in the variational form (\ref{eqn.fems1}). Based on the regularity estimates in (\ref{eqn.eta}) and (\ref{eqn.regani}) the solution is in $H^2$ in the sub-region of $\Omega$ that is away from the edge $e$ where the opening angle is $3\pi/2$. Therefore, a quasi-uniform mesh in such a region will yield a first-order (optimal) convergence for the interpolation error. In the neighborhood of the edge $e$, by (\ref{eqn.eta}) and Table 1 in \cite{CDN14}, we have 
\ben\label{eqn.prism1}
 u\in\mathcal M^2_{\bsigma+\bone}, \qquad {\rm{for}}\ \sigma_e<\eta_e=2/3\ {\rm{and}}\  \sigma_v<\eta_v=13/6,
\een 
where $\sigma_v$ is the index regarding the regularity of the solution near either of the vertices (endpoints of $e$) $v$ (see (\ref{eqn.newweight})). Then, by Theorem \ref{thm.mmm}, a sufficient condition to attain the optimal convergence rate for the finite element solution is that  the mesh parameters give rise to $a_e<2/3$ and $a_V<13/6$.

\begin{figure}
\includegraphics[scale=0.175]{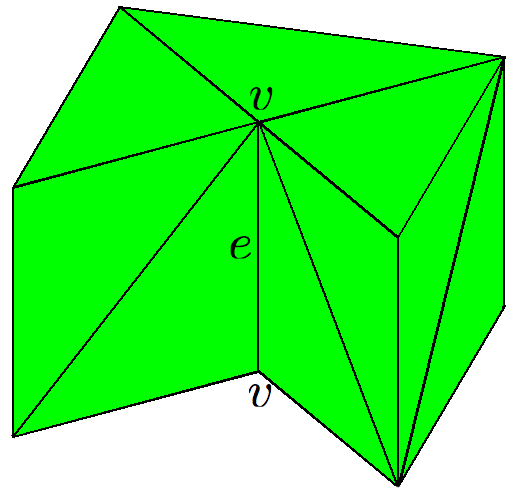}\hspace{2cm}
\includegraphics[scale=0.175]{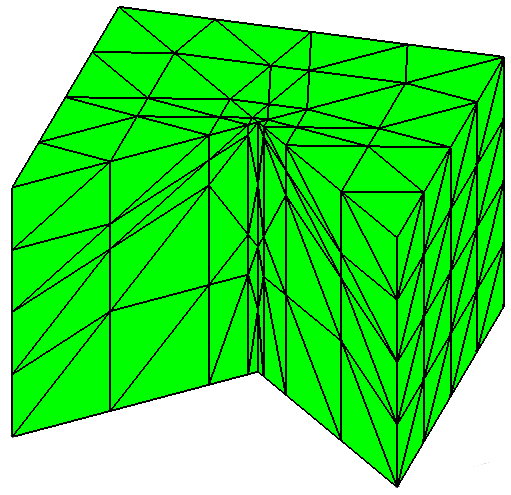}
\caption{{The prism domain: the initial triangulation (left) and the mesh after two graded refinements toward the singular edge $e$ ($\kappa_e=0.2$).}}
\label{fig.5}
\end{figure}

Recall the parameters $a_v, a_e\in(0, 1]$. Then, for linear elements, by (\ref{eqn.av}), we  have $a_V\leq 2-a_e<13/6$. Namely, the vertex $v$ shall not affect the convergence rate for any  feasible values  of $a_v$ and $a_e$, since   the regularity restriction for the vertex $v$ is always satisfied. Therefore, to improve the convergence rate, we only need to implement special edge refinement based on the value of $a_e$. Thus, in the numerical tests, we choose the parameters for the edge $e$ and for either of the vertices $v$, such that 
\ben\label{eqn.prism2} 0<a_e\leq 1 \qquad {\rm{and}}  \qquad a_v=1.
\een 
 Then, based on   Theorem \ref{thm.mmm}, in order to recover the optimal convergence rate  for the finite element solution, it suffices to choose $0<a_e<2/3$, namely, $0<\kappa_e=2^{-1/a_e}<0.353$. Recall that for $\kappa_{ev}=\kappa_e<0.5$ and $\kappa_v=2^{-1/a_v}=0.5$, the resulting mesh is graded toward the edge $e$ without special refinement for the vertex $v$. See Figure \ref{fig.5} for such graded meshes when $\kappa_e=0.2$. 

In Table \ref{tab.5.2new0}, we display the convergence rates of the finite element solution on proposed anisotropic meshes associated with different values of the grading parameter $\kappa_e$.   Here, $j$ is the level of refinements. Denote by $u_j$ the linear finite element solution on the mesh after $j$  refinements.  Since the exact solution is not known, the convergence rate is computed using the numerical solutions for successive mesh refinements 
\be\label{eqn.conv} {\rm{convergence\ rate}}=\log_2(\frac{|u_{j}-u_{j-1}|_{H^1(\Omega)}}{|u_{j+1}-u_{j}|_{H^1(\Omega)})}).
\ee 
As $j$ increases, the dimension of the discrete system is $O(2^{3j})$. Therefore, the asymptotic convergence rate in (\ref{eqn.conv}) is a reasonable indicator of the actual convergence rate for the numerical solution.

\begin{table}
\begin{center}
\begin{tabular}{|l|ll|}       \hline
\emph{$j$} & $\kappa_e= 0.1$\hspace{0.15cm} $\kappa_e= 
0.2$\hspace{0.15cm} $\kappa_e=  0.3$\hspace{0.15cm} $\kappa_e=  0.4$\hspace{0.15cm} $\kappa_e= 
0.5$&\\ \hline
\hspace{0.1cm}2 & 0.40 \hspace{.7cm} 0.46
\hspace{.7cm} 0.52 \hspace{.7cm} 0.58 \hspace{.7cm} 0.60  &\\\hline
\hspace{0.1cm}3 & 0.75 \hspace{.7cm} 0.79
\hspace{.7cm} 0.82 \hspace{.7cm} 0.84 \hspace{.7cm} 0.83  &\\\hline
\hspace{0.1cm}4 & 0.91 \hspace{.7cm} 0.93
\hspace{.7cm} 0.94 \hspace{.7cm} 0.93 \hspace{.7cm} 0.90  &\\\hline
\hspace{0.1cm}5 & 0.97 \hspace{.7cm} 0.98
\hspace{.7cm} 0.98 \hspace{.7cm} 0.96 \hspace{.7cm} 0.91  &\\\hline
\hspace{0.1cm}6 & 0.99 \hspace{.7cm} 0.99
\hspace{.7cm} 0.99 \hspace{.7cm} 0.97 \hspace{.7cm} 0.89  &\\\hline
\hspace{0.1cm}7 & 1.00 \hspace{.7cm} 1.00
\hspace{.7cm} 1.00 \hspace{.7cm} 0.97 \hspace{.7cm} 0.86  &\\\hline
\end{tabular}
\end{center}
\caption{Convergence rates for the prism domain.}
\label{tab.5.2new0}
\end{table}

It is clear from the table that the first-order convergence rate is obtained for $\kappa_e=0.1, 0.2,0.3<0.353$, while we lose the optimal convergence rate if $\kappa_e=0.4, 0.5$, both larger than the critical value $0.353$. When $\kappa_e=0.4$, that is $0.353<\kappa_e<0.5$, this choice still leads to an anisotropic mesh graded toward the singular edge, but the grading is insufficient to resolve the edge singularity in the solution, and hence does not lead to the optimal rate of convergence. These results are in strong agreement with the theoretical estimates in Section \ref{sec4}.

\begin{figure}
\includegraphics[scale=0.13]{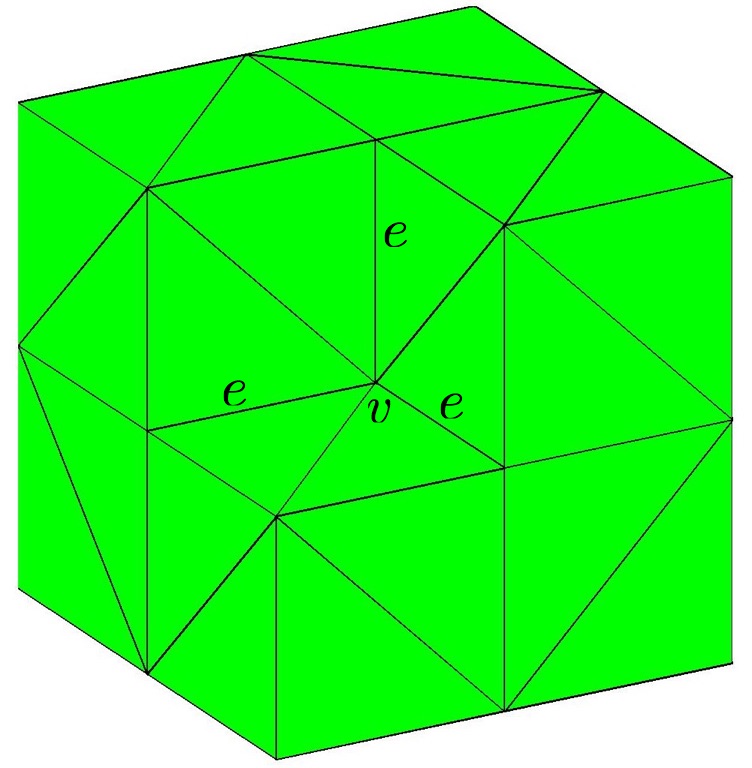}\hspace{2cm}
\includegraphics[scale=0.1]{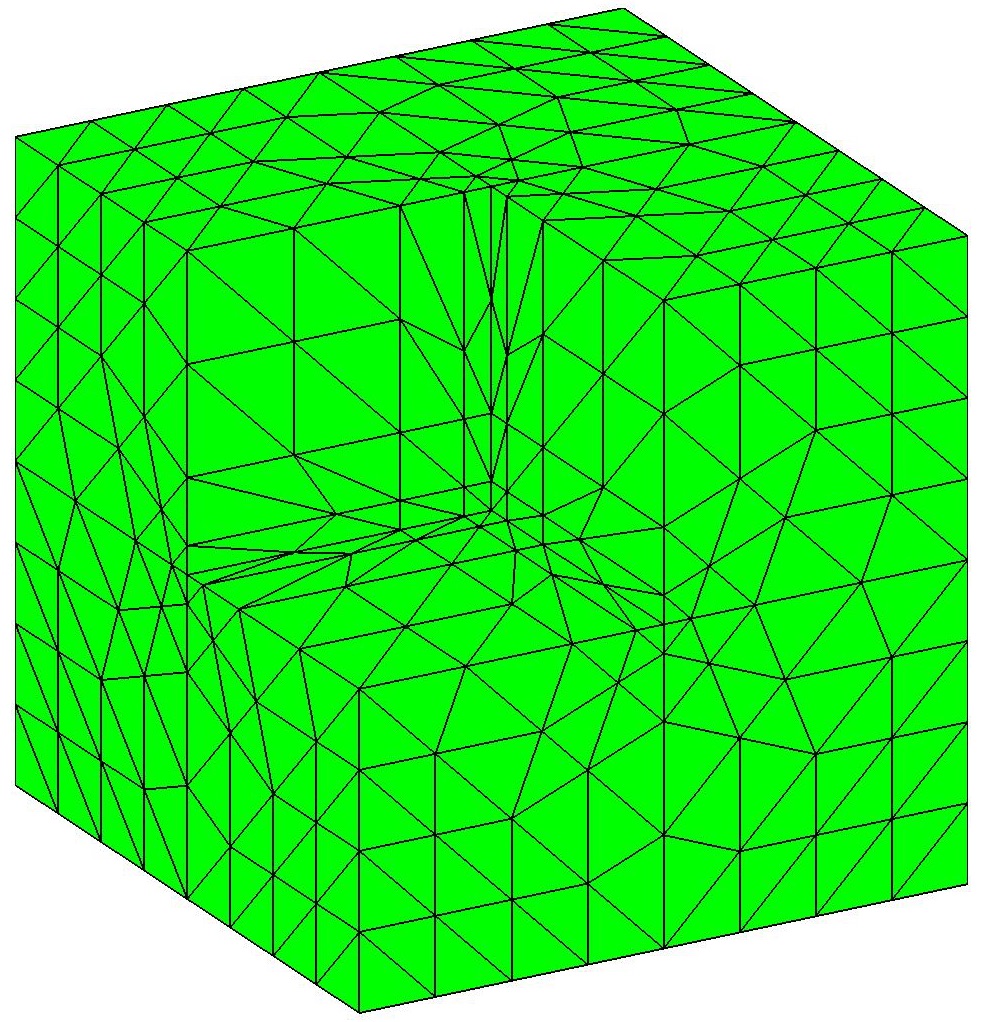}\hspace{2cm}
\includegraphics[scale=0.1]{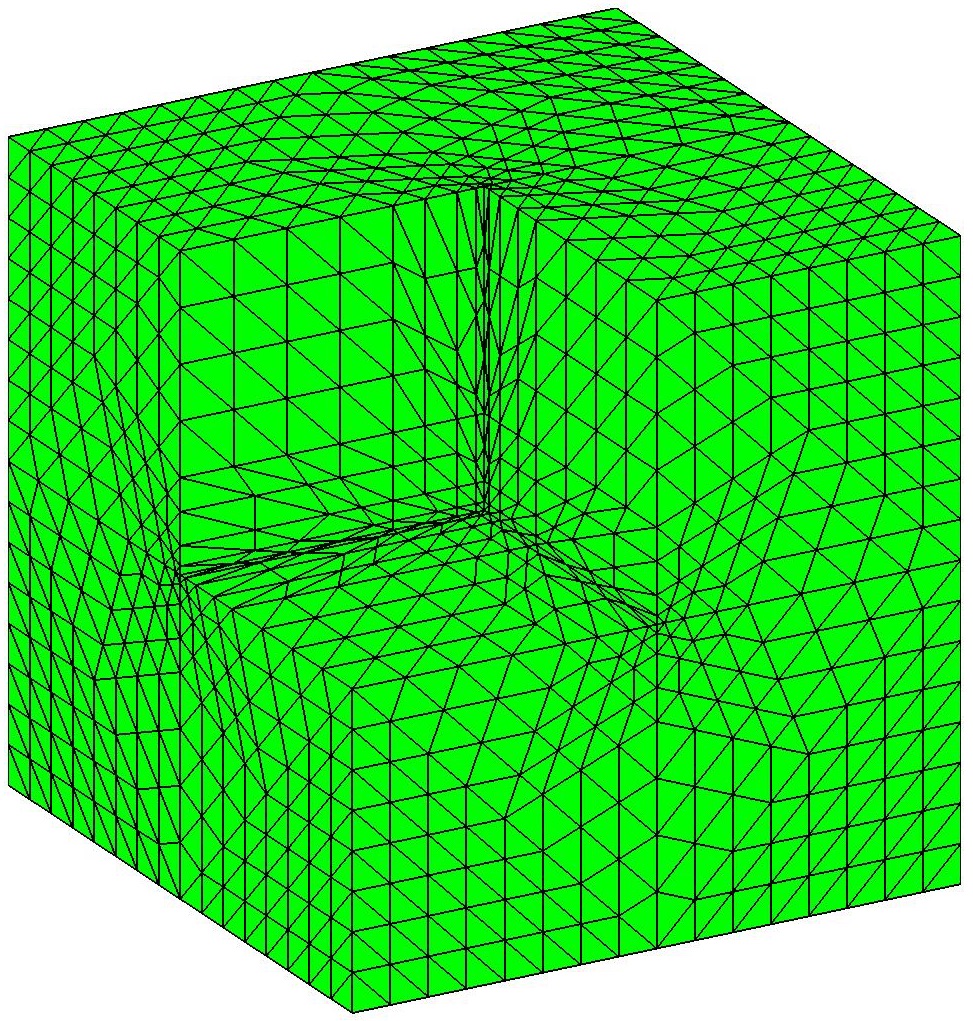}
\caption{{The Fichera corner (left -- right): the initial mesh, mesh after two refinements, mesh after three refinements ($\kappa_e=\kappa_v=0.3$).}}
\label{fig.6}
\end{figure}

\subsection{Test II} (The Fichera Corner) Let $D_0$  be the cube $(-1, 1)^3$ and  $D_1=[0, 1)^3$. Let the domain $\Omega:=D_0\setminus D_1$. Thus, the domain $\Omega$ is featured with the Fichera corner at the vertex $v$ and three adjacent edges $e$ with the opening angle $3\pi/2$ (Figure \ref{fig.6}). For a sub-region away from these three edges, the solution of equation (\ref{eqn.n1}) belongs to $H^2$, and therefore, a quasi-uniform mesh will lead to the optimal convergence rate for the interpolation error. In the neighborhood of the three edges, including the Fichera corner, by  (\ref{eqn.eta}), (\ref{eqn.regani}), and Table 1 in \cite{CDN14},  the solution satisfies
\ben\label{eqn.prism1}
 u\in\mathcal M^2_{\bsigma+\bone}, \qquad {\rm{for}}\ \sigma_e<\eta_e=2/3\ {\rm{and}}\  \sigma_v<\eta_v\approx 0.954.
\een 
For the endpoints of the three marked edges, which are not at the Fichera corner, the upper bound of the regularity index is $13/6$. For the same reason as in Test I, these vertices shall not affect the convergence rate for feasible mesh parameters. 
Then, by Theorem \ref{thm.mmm}, the sufficient condition to attain the optimal convergence rate for the finite element solution is that the mesh parameters give rise to $a_e < 2/3$ for the three marked edges and $a_V <0.954$ for the vertex $v$. There are many possible values of $a_e$ and $a_v$ that fulfill this requirement. To illustrate our method, in Table \ref{tab.5.2new1}, we list the convergence rates of the finite element solutions on  anisotropic meshes with $a_e=a_v=0.576$ (accordingly, $\kappa_e=\kappa_v=0.3$) and on quasi-uniform meshes $a_e=a_v=1$ (accordingly, $\kappa_e=\kappa_v=0.5$). The rates are computed using numerical solutions as in (\ref{eqn.conv}).

\begin{table}
\begin{tabular}{|l|l|l|}       \hline
\emph{$j$} & $\kappa_v= 0.3$\hspace{0.15cm} $\kappa_e= 
0.3$\hspace{0.15cm} &$\kappa_v=  0.5$\hspace{0.15cm} $\kappa_e=  0.5$\hspace{0.15cm}\\ \hline
\hspace{0.1cm}2 & \hspace{1.15cm}0.64 & \hspace{1.15cm}0.68  \\\hline
\hspace{0.1cm}3 &  \hspace{1.15cm}0.84& \hspace{1.15cm}0.82  \\\hline
\hspace{0.1cm}4 &  \hspace{1.15cm}0.94 & \hspace{1.15cm}0.86  \\\hline
\hspace{0.1cm}5 &  \hspace{1.15cm}0.97 & \hspace{1.15cm}0.86 \\\hline
\hspace{0.1cm}6 &  \hspace{1.15cm}0.99 & \hspace{1.15cm}0.83  \\\hline
\hspace{0.1cm}7 &  \hspace{1.15cm}0.99 & \hspace{1.15cm}0.80  \\\hline
\end{tabular}
\caption{Convergence rates for the Fichera corner.}
\label{tab.5.2new1}
\end{table}

In the case $\kappa_e=\kappa_v=0.3$, by (\ref{eqn.av}), we have $a_e=0.576<2/3$ and $a_V=a_v=0.576<0.954$. Therefore, by Theorem \ref{thm.mmm}, we expect to obtain the first-order optimal convergence rate in the finite element approximation. As for the quasi-uniform mesh ($\kappa_e=\kappa_v=0.5$), since the solution is not globally in $H^2$, by (\ref{eqn.subo}), we expect a sub-optimal convergence rate. It is clear that the numerical results in Table \ref{tab.5.2new1}  validate this theoretical prediction and hence verify the theory.

\protect\bibliographystyle{abbrv}
    \protect\bibliographystyle{alpha}
\def\cprime{$'$} \def\ocirc#1{\ifmmode\setbox0=\hbox{$#1$}\dimen0=\ht0
  \advance\dimen0 by1pt\rlap{\hbox to\wd0{\hss\raise\dimen0
  \hbox{\hskip.2em$\scriptscriptstyle\circ$}\hss}}#1\else {\accent"17 #1}\fi}

\end{document}